\documentclass[11pt]{article}
\usepackage{amsmath}
\usepackage{amsfonts}
\usepackage{amsthm}
\usepackage{amssymb, setspace}
\usepackage{color,graphicx,epsfig,geometry,hyperref,fancyhdr}
\usepackage[T1]{fontenc}
\usepackage{float}
\usepackage{subfig}
\usepackage{commath}
\usepackage{bbm}
\usepackage{mathrsfs,fleqn}
\usepackage{mathtools}
\newtheorem{theorem}{Theorem}[section]

\newtheorem{lemma}{Lemma}[section]

\newcommand{\Z}{\mathbb{Z}}

\newcommand{\R}{\mathbb{R}}
\newcommand{\C}{\mathbb{C}}

\newcommand{\grad}{\nabla}

\newcommand{\brac}[1]{\left\{#1\right\}}

\graphicspath{{paper-pics/}}

\begin{document}

\begin{flushleft}
\Large 

\noindent{\bf \Large Direct and inverse scattering for an isotropic medium with a second-order boundary condition}

\end{flushleft}

\vspace{0.2in}

{\bf  \large Govanni Granados}\\
\indent {\small Department of Mathematics, University of North Carolina at Chapel Hill,\\ \indent Chapel Hill, NC 27599, United States of America }\\
\indent {\small Email: \texttt{ggranad@unc.edu}}\\

{\bf  \large Isaac Harris}\\
\indent {\small Department of Mathematics, Purdue University, West Lafayette, IN 47907, United States of \\ \indent America }\\
\indent {\small Email: \texttt{harri814@purdue.edu} }\\

{\bf  \large Andreas Kleefeld}\\
\indent {\small Forschungszentrum J\"{u}lich GmbH, J\"{u}lich Supercomputing Centre, } \\
\indent {\small Wilhelm-Johnen-Stra{\ss}e, 52425 J\"{u}lich, Germany}\\
\indent {\small University of Applied Sciences Aachen, Faculty of Medical Engineering and } \\
\indent {\small Technomathematics, Heinrich-Mu\ss{}mann-Str. 1, 52428 J\"{u}lich, Germany}\\
\indent {\small Email: \texttt{a.kleefeld@fz-juelich.de}}\\


\begin{abstract}
\noindent 
We consider the direct and inverse scattering problem for a penetrable, isotropic obstacle with a second-order Robin boundary condition, which asymptotically models the delamination of the boundary of the scatterer. We develop a direct sampling method to solve the inverse shape problem by numerically recovering the scatterer. Here, we assume that the corresponding Cauchy data is measured on the boundary of a region that fully contains the scatterer. Similar methods have been applied to other inverse shape problems, but they have not been studied for a penetrable, isotropic scatterer with a second-order Robin boundary condition. We also initiate the study of the corresponding transmission eigenvalue problem, which is derived from assuming zero Cauchy data is measured on the boundary of the region that fully contains the scatterer. We prove that the transmission eigenvalues for this problem are at most a discrete set. Numerical examples will be presented for the inverse shape problem in two dimensions for circular and non-circular scatterers. Further, transmission eigenvalues are computed numerically for various scatterers.
\noindent 
\end{abstract}



\section{Introduction}
Inverse scattering has played a major role within a series of imaging modalities as it arises in many practical applications such as medical imaging \cite{tamori1}, nondestructive testing \cite{deng1, kharkovsky1}, and geophysical exploration \cite{sato1}. In this paper, we are interested in the inverse scattering problem for an isotropic material with a delaminated boundary. This can be seen as a coating on the boundary than models a thin layer made out of a different material from the  scatterer. This thin layer will be asymptotically modeled as an interface where a Generalized Impedance Boundary Condition is given. This class of boundary conditions are well known for modeling imperfectly conducting and impenetrable obstacles as well as thin coatings in inverse scattering \cite{bendali, durufle, haddar1, haddar2}.  Previously, \cite{cakoni1} models the delamination of a thin region on a portion of the boundary of a region within the scatterer. In this paper, the delamination occurs on the entire boundary of the scatterer. It will be modeled by a second-order (Robin) boundary condition characterized by the jump of the trace of the normal derivative of the scattered field across the entire boundary of a penetrable obstacle. Recently, this boundary condition has also been studied in a shape identification problem coming from electrostatics \cite{EIT-granados2}.

To solve this inverse scattering problem with fixed frequency, we will employ a qualitative method. One of their main advantages is that they generally require little a priori knowledge about the scatterer, which makes them advantageous for nondestructive testing. In this paper, we will extend the applicability of the well known Direct Sampling Method by developing a stable reconstruction algorithm given Cauchy data on a measurement boundary. This particular qualitative method has been studied for inverse scattering for an impenetrable obstacle \cite{griesmaier} and for an isotropic scatterer \cite{Ito2012,Ito2013,LiZou}. Indeed, this method works for other imaging modalities like as seen in \cite{ChowHanZou,ChowItoLiuZou}. Recently, it has been utilized for a scattering problem modeled by the biharmonic wave equation \cite{harris-lee-li}. On a broader scope, different classes of qualitative methods, such as Linear Sampling \cite{arens,GLSM, EM-cakoni, cheney1, garnier}, Factorization Method \cite{bondarenko, harris4, kirsch}, and the Multiple Signal Classification algorithm \cite{MUSIC-ammari-scattering, MUSIC-jake,MUSIC-EM1} have also been studied for various types of inverse scattering problems. 

In the case of the Linear Sampling and Factorization Methods, it is well-known that these methods are not valid at the so-called transmission eigenvalues \cite{cakoni-colton,analytic-fred}. These values can be seen as corresponding to specific wave numbers where the scattered field is zero outside the scatterer. Therefore, one would not be able to detect the scatterer if there were such an incident field. The transmission eigenvalues can, in some sense, be seen as a non-scattering frequency when such an incident field exists. In general, the transmission eigenvalue problems derived from inverse scattering are non-self-adjoint and nonlinear. This means studying them is mathematically challenging but also interesting since they often depend monotonically on the material parameters \cite{Hughes}. Here we begin the study of the transmission eigenvalue problem corresponding to our scattering problem, which we derive from assuming zero Cauchy data of the scattered field on the measurement boundary. This problem is more complicated to handle than other transmission eigenvalue problems due to the second-order boundary condition. 

The rest of the paper is organized as follows. In Section \ref{dp-ip}, we introduce in full detail the direct scattering problem for an isotropic medium with a second-order boundary condition. We will use a variational method to prove well-posedness for the direct problem and derive the appropriate functional settings. In Section \ref{inv-scatter}, we derive and analyze an imaging functional via direct sampling. This method requires using incident plane waves from a finite number of incident directions. The induced, measured Cauchy data of the scattered field is used to derive a stable imaging functional with respect to noise. In Section \ref{numerical-section}, numerical examples are presented in $\mathbb{R}^2$ to validate the analysis of the imaging functional. In Section \ref{discrete-TEV} we consider the associated Transmission Eigenvalue problem. By appealing to the analytic Fredholm Theorem \cite{analytic-fred}, we show that the transmission eigenvalues form an at most discrete set in the complex plane for a specific case of the refractive index. Then, in Section \ref{numerical-TEV}, we present numerical examples for different scatterers. Conclusion and comments on future work are given in Section \ref{conclusion}.	 

\section{The Direct Scattering Problem}\label{dp-ip}
We begin by considering the direct problem associated with the inverse scattering and the transmission eigenvalue problem with a generalized Robin transmission condition on its boundary. Let $D \subset \mathbb{R}^{d}$ for $d = 2$ or $3$ be a simply connected open set with Lipschitz boundary $\partial D$ where $\nu$ is the unit outward normal vector. We assume that the refractive index $n \in L^{\infty}(D)$ is complex valued such that
\begin{equation}\label{refraction-index}
    \text{Im}(n) \geq 0,
\end{equation}
where the support of $n-1$ is $D$, i.e. $n =1$ in $\mathbb{R}^d \backslash \overline{D}$. We will focus on the case where the Robin boundary condition on $\partial D$ asymptotically models delamination between two materials. The total field is denoted by $u (x) = u^i (x) + u^s (x)$, with $u^s$ denoting the scattered field created and $u^{i} (x,\hat{y}) = \text{e}^{\text{i}kx \cdot \hat{y}}$ with $|\hat{y}|=1$ denotes the incident plane wave that illuminates the scatterer. The direct scattering problem for an isotropic scatterer with delamination is given by: find $u^s$ which satisfies
\begin{equation} \label{bvp}
\Delta u^s + k^2 n u^s = - k^2 (n-1) u^i \enspace \text{in} \enspace \mathbb{R}^{d} \backslash \partial D \quad \text{with} \quad [\![ u ]\!] \big \rvert_{\partial D} = 0
\end{equation}
as well as the Robin boundary condition
\begin{equation} \label{2nd-bc}
[\![\partial_\nu u^s ]\!] \big \rvert_{\partial D} = \mathscr{B}(u) \quad \text{where} \quad  [\![\partial_\nu u^s ]\!] \big|_{\partial D} := ( \partial_{\nu} u_{+} - \partial_{\nu} u_{-}) \big|_{\partial D}.
\end{equation}
The `+' notation represents the trace taken from $B_R \setminus \overline{D}$ and the `$-$' notation represents the trace taken from $D$.
Here, $ \mathscr{B}( \cdot )$ is a so--called Laplace--Beltrami boundary operator that is defined as 
$$\mathscr{B}(u) := -\nabla_{\partial D} \cdot \mu \nabla_{\partial D} u + \gamma u .$$
In the $\mathbb{R}^2$ case, the operator $ \nabla_{\partial D} \cdot \mu \nabla_{\partial D}$ is replaced by the operator $\frac{\text{d}}{\text{d}s} \mu \frac{\text{d}}{\text{d}s}$ where ${\text{d}} / \text{d}s$ is the tangential derivative and $s$ is the arc-length. This second-order Robin condition in \eqref{2nd-bc} models the delamination of the scatterer $D$ on its boundary $\partial D$ and states that the jump of the normal derivative of the scattered field $\partial_{\nu} u^s$ across this boundary is quasi-proportional to the total field $u$. The scattered field satisfies the radiation condition 
$$\partial_{r} u^s - \text{i}ku^s = \mathcal{O} \bigg( \dfrac{1}{r^{(d+1)/2}} \bigg) \,\, \text{as} \,\, r \rightarrow \infty \quad \text{uniformly with respect to $\hat{x} = \frac{x}{|x|}$ where $r = |x|$.}$$ 

As in other works, we consider \eqref{bvp} on the truncated domain $B_R := \{ x \in \mathbb{R}^d \, : \, |x| < R\}$ for a fixed $R>0$ such that $D \subset B_R$ and dist$(\partial B_R , \overline{D}) > 0$. Due to the second-order Robin condition \eqref{2nd-bc}, the scattered field $u^s$ that satisfies \eqref{bvp} requires half an order of additional regularity on the boundary $\partial D$. Thus, the appropriate solution space is the Hilbert space defined as
$$ \widetilde{H}^{1}(B_R) := \brac{ \varphi \in H^{1}(B_R) \quad \text{such that} \quad \varphi \big \rvert_{\partial D} \in H^{1}(\partial D) }.$$ 
This will be equipped with the norm
$$\norm{\varphi}^{2}_{\widetilde{H}^{1}(B_R)} = \norm{\varphi}^{2}_{H^{1}(B_R)} + \norm{\varphi}^{2}_{H^{1}(\partial D)}. $$
Since $u^s \in \widetilde{H}^{1}(B_R) $, it is known that $[\![ u ]\!] \big \rvert_{\partial D} = 0$. For analytical purposes of well-posedness of the direct problem and the analysis of the inverse and transmission eigenvalue problem in Section \ref{inv-scatter} and Section \ref{discrete-TEV}, we make the following sign assumptions on the boundary parameters $\gamma \in L^{\infty} (\partial D)$ and $\mu \in L^{\infty} \big(\partial D , \C^{(d-1)\times (d-1)}\big)$. Namely, we assume that there exists positive constant $\gamma_{\text{min}} > 0$ such that the real- and imaginary-part of the coefficient $\gamma$ satisfies
\begin{equation}\label{gamma-bounds}
\text{Re}(\gamma) \geq \gamma_{\text{min}} > 0 \quad \text{and} \quad - \text{Im}(\gamma) \geq 0
\end{equation}
for almost every $x \in \partial D$. We also assume that the real- and imaginary-parts of $\mu$ are Hermetian definite matrices where there exists positive constant $\mu_{\text{min}} > 0$ such that
\begin{equation}\label{mu-bounds}
 \overline{\xi}\cdot \text{Re}(\mu ) \xi \geq \mu_{\text{min}}  |\xi|^2 \quad \text{and} \quad - \overline{\xi} \cdot \text{Im}(\mu )  \xi \geq 0
\end{equation}
for all $\xi \in \mathbb{C}^{d-1}$ for almost every $x \in \partial D$. We now consider Green's 1st Theorem on the region $B_R \backslash \overline{D}$ 
$$\int_{B_R \setminus \overline{D}} \nabla u^s \cdot \nabla \overline{\varphi} -k^2 n u^s \overline{\varphi} \, \text{d}x - \int_{\partial B_R} \overline{\varphi} \partial_{\nu} u^s \, \text{d}s + \int_{\partial D} \overline{\varphi} \partial_{\nu} u^{s}_{+} \, \text{d}s  = 0, $$
as well as Green's 1st Theorem on the region $D$
$$ \int_{D} \nabla u^s \cdot \nabla \overline{\varphi} -k^2 n u^s \overline{\varphi}  \, \text{d}x  - \int_{\partial D} \overline{\varphi} \partial_{\nu} u^{s}_{-} \, \text{d}s  = \int_{D} \overline{\varphi} k^2 (n-1) u^i \, \text{d}x ,$$ 
for any test function $\varphi \in \widetilde{H}^1 (B_R)$. The variational formulation for \eqref{bvp} is given by  adding these two equations

\begin{align}\label{vf}
\int_{B_R} \nabla u^s \cdot \nabla \overline{\varphi} \, \text{d}x &-k^2 n u^s \overline{\varphi} \, \text{d}x + \int_{\partial D}  \mu \nabla_{\partial D} u^s \cdot \nabla_{\partial D} \overline{\varphi}  +\gamma  u^s \overline{\varphi} \, \text{d}s - \int_{\partial B_R} \overline{\varphi} \partial_{\nu} u^s \, \text{d}s \nonumber \\
&  = \int_{D} \overline{\varphi} k^2 (n-1) u^i \, \text{d}x - \int_{\partial D}  \mu \nabla_{\partial D} u^i \cdot \nabla_{\partial D} \overline{\varphi} - \gamma \overline{\varphi} u^i \, \text{d}s
\end{align}
where we have used the second-order Robin transmission condition  \eqref{2nd-bc} on $\partial D$. 

We prove the well-posedness of \eqref{bvp}--\eqref{2nd-bc} by using a variational technique, namely, the Fredholm Alternative. We write the variational formulation \eqref{vf} as an equivalent problem by introducing the Dirichlet-to-Neumann
operator (DtN) $\Lambda : H^{1/2}(\partial B_R) \rightarrow H^{-1/2}(\partial B_R)$ defined as 
\begin{equation}\label{dtn}
\Lambda f = \partial_{\nu} w \big \rvert_{\partial B_R} \quad \text{such that } \quad \text{$\Delta w + k^2 w = 0$ and  $w = f$ on $\partial B_R$} 
\end{equation}
where $w$ satisfies the radiation condition. We will need the following important property of the DtN operator as shown in Theorem 5.22 in \cite{cakoni-colton}.

\begin{lemma}\label{dtn_0}
The DtN map $\Lambda$ is a bounded linear operator from $H^{1/2}(\partial B_R)$ to $H^{-1/2}(\partial B_R)$. Furthermore, there exists a bounded operator $\Lambda_0: H^{1/2}(\partial B_R) \rightarrow H^{-1/2}(\partial B_R)$ satisfying
$$- \int_{\partial B_R}  \overline{w}  \Lambda_0 w \, \text{d}s \geq C \norm{w}^{2}_{H^{1/2}(\partial B_R)} $$ for some constant $C>0$, such that $\Lambda - \Lambda_0 :  H^{1/2}(\partial B_R) \rightarrow H^{-1/2}(\partial B_R) $ is compact.
\end{lemma}

We write the variational formulation \eqref{vf} as 
$$a(u^s , \varphi ) = \ell(\varphi ) \quad \text{for all} \quad \varphi \in \widetilde{H}^{1}(B_R)$$
where $a : \widetilde{H}^{1}(B_R) \times \widetilde{H}^{1}(B_R) \rightarrow \mathbb{C}$ is a sesquilinear form defined as
\begin{align}\label{a}
a (u^s , \varphi ) = \int_{B_R} \nabla u^s \cdot \nabla \overline{\varphi}  &- k^2 n u^s  \overline{\varphi} \, \text{d}x \nonumber  \\
&+ \int_{\partial D}  \mu \nabla_{\partial D} u^s \cdot \nabla_{\partial D} \overline{\varphi}  + \gamma  u^s \overline{\varphi} \, \text{d}s - \int_{\partial B_R} \overline{\varphi} \Lambda u^s \, \text{d}s
\end{align}
and $\ell: \widetilde{H}^{1}(B_R) \rightarrow \mathbb{C}$ is a bounded conjugate linear form defined as
\begin{equation}
\ell (\varphi) = k^2 \int_{D}  (n-1) u^i \overline{\varphi} \, \text{d}x - \int_{\partial D}  \mu \nabla_{\partial D} u^i \cdot \nabla_{\partial D} \overline{\varphi} + \gamma u^i  \overline{\varphi} \, \text{d}s.
\end{equation}
From \eqref{a}, we write $a = a_1 + a_2$, where
$$a_1 (u^s , \varphi ) = \int_{B_R} \nabla u^s \cdot \nabla \overline{\varphi} \, \text{d}x - \int_{\partial B_R} \overline{\varphi} \Lambda_0 u^s \, \text{d}s + \int_{\partial D}  \mu \nabla_{\partial D} u^s \cdot \nabla_{\partial D} \overline{\varphi}  + \gamma  u^s \overline{\varphi} \, \text{d}s$$
and 
$$ a_2 (u^s , \varphi ) = - k^2 \int_{B_R} n u^s \overline{\varphi} \, \text{d}s - \int_{\partial B_R} \overline{\varphi} (\Lambda - \Lambda_0)u^s \, \text{d}s$$
where $\Lambda_0$ is the operator defined in Lemma \ref{dtn_0}. By the boundedness of $\Lambda_0$ and the Trace Theorem, $a_1$ is bounded. Then by the Riesz Representation Theorem, there exists a bounded linear operator $A_1 :  \widetilde{H}^{1}(B_R) \rightarrow \widetilde{H}^{1}(B_R)$ such that
$$ a_1 (u^s , \varphi ) = (A_1 u^s , \varphi)_{\widetilde{H}^{1}(B_R)} \quad \text{for all} \quad \varphi \in \widetilde{H}^{1}(B_R).$$
Furthermore, by the sign assumptions \eqref{gamma-bounds} and \eqref{mu-bounds} on the boundary coefficients and the estimate from Lemma \ref{dtn_0} we have that
\begin{align*}
	|a_1(u^s , u^s)|  &\geq  \norm{ \nabla u^s}^{2}_{L^{2}(B_R)} - \int_{\partial B_R} \overline{u^s} \Lambda_0 u^s \, \text{d}s + \int_{\partial D}  \text{Re}(\mu ) | \nabla_{\partial D} u^s |^2 + \text{Re}(\gamma ) | u^s |^2 \, \text{d}s \\
	&\geq \norm{ \nabla u^s}^{2}_{L^{2}(B_R)} + C \norm{u^s}^{2}_{H^{1/2}(\partial B_R)} +\text{min} \{ \gamma_{\text{min}} ,  \mu_{\text{min}} \} \norm{ u^s}^{2}_{H^1 (\partial D )} \\
	&\geq \alpha \norm{ u^s}^{2}_{\widetilde{H}^{1}(B_R)}
\end{align*} 
where $C>0$ is the constant from Lemma \ref{dtn_0}, and $\alpha>0$ is a positive constant. Thus, $a_1$ is coercive and the Lax-Milgram Lemma implies that the operator $A_1:  \widetilde{H}^{1}(B_R) \rightarrow \widetilde{H}^{1}(B_R)$ has a bounded inverse. Once again, by the Riesz Representation Theorem, there exists a bounded linear operator $A_2: \widetilde{H}^{1}(B_R) \rightarrow \widetilde{H}^{1}(B_R)$ such that
$$a_2 (u^s , \varphi ) = (A_2 u^s , \varphi)_{\widetilde{H}^{1}(B_R)} \quad \text{for all} \quad \varphi \in \widetilde{H}^{1}(B_R).$$
By the compactness of $\Lambda - \Lambda_0$ and the compact embedding of $H^1 (B_R)$ into $L^2 (B_R)$, it follows that $A_2$ is compact. Lastly, by the Riesz Representation Theorem again, there exists a unique $f \in \widetilde{H}^{1}(B_R)$ such that
$$ \ell (\varphi ) = (f , \varphi)_{\widetilde{H}^{1}(B_R)}\quad \text{for all} \quad \varphi \in \widetilde{H}^{1}(B_R).$$
Thus, the variational formulation \eqref{vf} is equivalent to the following problem, which is an operator equation:
\begin{equation}
(A_1 + A_2) u^s = f \quad \text{ where } \quad \| f\|_{\widetilde{H}^{1}(B_R)} \leq C \left\{ \| u^i \|_{{H}^{1}(D)} + \| u^i \|_{{H}^{1}(\partial D)} \right\}.
\end{equation}
Notice that, since $A_1$ is invertible and $A_2$ is compact this implies that the scattering problem \eqref{bvp}--\eqref{2nd-bc} is Fredholm of index zero, i.e. the problem is well--posed provided we have uniqueness. To this end, we now show that the solution to our scattering problem is unique. 

\begin{theorem}
The direct scattering operator problem \eqref{bvp}--\eqref{2nd-bc} has at most one solution. Moreover, there is a unique solution to  \eqref{bvp}--\eqref{2nd-bc} satisfying 
$$\| u^s  \|_{\widetilde{H}^{1}(B_R)} \leq C \left\{ \| u^i \|_{{H}^{1}(D)} + \| u^i \|_{{H}^{1}(\partial D)} \right\}. $$  
\end{theorem}
\begin{proof}
In order to prove the claim, by equivalence to the variational formulation and linearity it is sufficient to show that Null$(A_1 + A_2)=\{0\}$, i.e. the incident field $u^i = 0$ implies that $u^s=0$. Therefore, by taking $u^i=0$ and letting $v = u^s$ in \eqref{vf} we get 
$$ \text{Im} \int_{\partial B_R} \overline{u^s} \Lambda u^s \, \text{d}s =  \int_{\partial D} \text{Im}(\mu) \rvert \nabla_{\partial D} u^s \rvert^{2} + \text{Im}(\gamma) \rvert u^s \rvert^{2} \, \text{d}s - k^2 \int_{B_R} \text{Im}(n) \rvert u^s \rvert^{2} \, \text{d}x\leq 0$$
by the positivity assumption of on the refraction index and the sign assumptions on the boundary coefficients. By appealing to Theorem 3.6 in \cite{cakoni-colton} we have that $u^s = 0$ in $\mathbb{R}^d \setminus \overline{D}$. So on the boundary of the scatterer $\partial D$ we have that 
$$ 0= u^s_{+} \big \rvert_{\partial D} = u^s_{-} \big \rvert_{\partial D} \quad \text{ since } \quad  [\![ u^s]\!] \big \rvert_{\partial D} = 0 $$
and  $0=  \partial_\nu u^s_{+} \big \rvert_{\partial D} = \partial_{\nu} u^s_{-} \big \rvert_{\partial D}$ by the boundary condition $[\![ \partial_{\nu} u^s]\!] \big \rvert_{\partial D} = \mathscr{B}(u^s)$ on $\partial D$. This implies that in the scatterer $D$ the scattered field $u^s$ satisfies
$$\Delta u^s + k^2 n u^s = 0 \enspace \text{in} \enspace D \quad \text{with} \quad u^s_{-} \big \rvert_{\partial D} = \partial_{\nu} u^s_{-} \big \rvert_{\partial D} =0.$$
By Holmgren's Theorem (see for e.g. \cite{holmgren}), we have that $u^{s} = 0$ in $D$, proving the claim. 
\end{proof}

The above result proves that \eqref{bvp} with the boundary condition \eqref{2nd-bc} is well--posed by the Fredholm Alternative. {\color{black} Now that we know that the direct problem is well--posed, we can consider the inverse shape problem of recovering the scatterer $D$ from measured scattering data.} In the following section, we derive a direct sampling method for recovering the unknown scatterer from the measured Cauchy data.

\section{The Inverse Scattering Problem}\label{inv-scatter}
In this section, we focus on defining an imaging functional that will reconstruct the scatterer $D$. We define the incident field as the plane wave
$$u^{i} (x,\hat{y}) = \text{e}^{\text{i}kx \cdot \hat{y}}, \enspace 	x \in \mathbb{R}^{d}, \enspace \hat{y} \in \mathbb{S}^{d-1} := \brac{x \in \mathbb{R}^d \enspace \text{such that} \enspace |x|=1}, $$
where $\hat{y}$ is the direction vector of propagation. Thus, the total and scattered field will also depend on $\hat{y}$, i.e. $u(x,\hat{y}) = u^{i}(x,\hat{y}) + u^{s}(x,\hat{y})$. We also consider a domain $\Omega \subset \mathbb{R}^d$ such that $\overline{D} \subset \Omega$. The boundary of this domain $\partial\Omega$ is the measurements curve/surface where we collect the Cauchy data for multiple incident directions. Indeed, we assume that $u^{s}(\cdot , \hat{y}_j)$ and $\partial_{\nu}u^{s}(\cdot , \hat{y}_j) $ on $\partial \Omega$ is given for $\hat{y}_j \in \mathbb{S}^{d-1}$ where $j=1,\hdots,N$. With this, we will determine the scatterer $D$. We now define the free-space Green's function for the Helmholtz equation $\Phi(x,y)$ as
$$ \Phi(x,y)  =   \dfrac{\text{i}}{4} H^{(1)}_{0}(k \rvert x - y \rvert ) \, \, \text{in $\mathbb{R}^2$} \quad \text{ and }  \quad \Phi(x,y)  =  \frac{ \text{e}^{ \text{i}k \rvert x- y \rvert}}{4 \pi \rvert x- y \rvert}  \, \, \text{in $\mathbb{R}^3$}.$$ 
Here we assume $x \neq y$ with $H^{(1)}_{0}$ being the zeroth order Hankel function of the first kind. We now study a function that will serve as the foundation for our imaging functional to reconstruct the scatterer. This function is motivated by the imaging functionals discussed in \cite{rgf-granados1,harris-le,lan-nguyen,sampling-with-deeplearning}. For sampling points $z \in \mathbb{R}^d$ we, define the imaging functional to be studied for the inverse scattering problem $W(z) : \mathbb{R}^d \rightarrow \mathbb{R}_{\geq 0}$ as
\begin{equation}\label{wz-raw}
W(z) = \sum_{j=1}^{J} \bigg \rvert \int_{\partial \Omega} u^s (x , \hat{y}_j) \partial_{\nu (x)} \text{Im}\Phi(x,z) - \partial_{\nu (x)}  u^s (x , \hat{y}_j) \text{Im}\Phi(x,z) \, \text{d}s(x) \bigg \rvert ,
\end{equation}
where again $u^s (x , \hat{y}_j)$ is the scattered field satisfying \eqref{bvp}--\eqref{2nd-bc} with incident field $u^{i} (x,\hat{y}) = \text{e}^{\text{i}kx \cdot \hat{y}}$. It is well known that 
\begin{equation}\label{im-phi}
\text{Im}\Phi(x,y) =  \dfrac{1}{4} J_{0}(k \rvert x - y \rvert ) \, \, \text{in $\mathbb{R}^2$} \quad \text{ and }  \quad \text{Im}\Phi(x,y) = \frac{k}{4 \pi} j_{0}(k \rvert x - y \rvert )   \, \, \text{in $\mathbb{R}^3$}
\end{equation}
where $J_0$ and $j_0$ are the Bessel function and spherical Bessel function of the first kind, respectively. The following result demonstrates the dependence of this imaging functional on the scatterer. More specifically, it details how measured surface data, namely $u^s$ and $\partial_{\nu} u^s$ on $\partial \Omega$, are connected to the scatterer $D$. {\color{black} This connection will make the dependance of the imaging functional $W(z)$ on the scatterer explicit. Indeed, it is clear that $W(z)$ depends on the scatterer but it is not clear how it depends on $D$. The following result will allow use to see the dependance explicitly which can be exploited to show that plotting $W(z)$ can recover $D$. }

\begin{lemma}\label{wz} 
The imaging functional $W(z)$ defined in \eqref{wz-raw} satisfies 
\begin{equation*}\label{wz-d}
W(z) = \sum_{j=1}^{J} \bigg \rvert \int_{D} k^2 (n(x)-1) \mathrm{Im} \Phi (x,z) u(x, \hat{y}_j) \, \mathrm{d}x - \int_{\partial D} \mathrm{Im} \Phi (x,z) \mathscr{B}(u(x, \hat{y}_j)) \, \mathrm{d}s(x) \bigg \rvert
\end{equation*}
where $u (x , \hat{y})$ is the total field corresponding to \eqref{bvp}--\eqref{2nd-bc} with incident field $u^{i} (x,\hat{y}) = \mathrm{e}^{\mathrm{i}kx \cdot \hat{y}}$. 
\end{lemma}

\begin{proof}
Notice, that due to the fact that both $u^s (x , \hat{y})$ and $\text{Im}\Phi(x,z)$ solve the Helmholtz equation in $\Omega \setminus \overline{D}$ we have that  
$$0= \int_{\Omega \setminus \overline{D}} u^s (x, \hat{y}) \Delta_{x} \text{Im} \Phi(x,z) - \text{Im} \Phi(x,z) \Delta_{x} u^s (x, \hat{y}) \, \text{d}x.$$ 
Now, by appealing to Green's second identity in $\Omega \setminus \overline{D}$ we have that 
\begin{multline*}
0 = \int_{\partial \Omega} u^{s}(x, \hat{y}) \partial_{\nu (x)} \text{Im} \Phi (x,z) -  \partial_{\nu (x)} u^{s}(x, \hat{y}) \text{Im} \Phi (x,z) \, \text{d}s(x) \\ - \int_{\partial D} u^{s}(x, \hat{y}) \partial_{\nu (x)} \text{Im} \Phi (x,z) -  \partial_{\nu (x)} u^{s}_{+}(x, \hat{y}) \text{Im} \Phi (x,z) \, \text{d}s(x).
\end{multline*}
In order to prove the claim, we now use the boundary condition \eqref{2nd-bc}  to obtain that 
\begin{align}
\int_{\partial \Omega} u^{s}(x, \hat{y})  \partial_{\nu (x)} \text{Im} \Phi (x,z) &-  \partial_{\nu (x)} u^{s}(x, \hat{y}) \text{Im} \Phi (x,z) \, \text{d}s(x)= \nonumber  \\
 &\int_{\partial D} u^{s}(x, \hat{y}) \partial_{\nu (x)} \text{Im} \Phi (x,z) -   \partial_{\nu (x)} u^{s}_{-}(x, \hat{y})  \text{Im} \Phi (x,z) \, \text{d}s(x) \nonumber  \\
  &\hspace{1.25in}  -  \int_{\partial D}\mathscr{B}(u(x, \hat{y}_j)) \text{Im} \Phi (x,z) \, \text{d}s(x). \label{rgf-equality1}
\end{align}
By Green's second identity in ${D}$ and \eqref{bvp} we see that 
\begin{align}
\int_{\partial D} u^{s}(x, \hat{y}) \partial_{\nu (x)} \text{Im} \Phi (x,z) &  -  \partial_{\nu (x)} u^{s}_{-}(x, \hat{y}) \text{Im} \Phi (x,z) \, \text{d}s(x)  \nonumber \\
 &=\int_{D} u^s (x, \hat{y}) \Delta_{x} \text{Im} \Phi (x,z) - \text{Im} \Phi (x,z) \Delta_{x} u^s (x, \hat{y}) \, \text{d}x  \nonumber \\
 &=\int_{D} k^2 (n(x)-1) \text{Im} \Phi (x,z) u(x, \hat{y}) \, \text{d}x. \label{rgf-equality2}
\end{align}
Therefore, by combining the equalities in \eqref{rgf-equality1}--\eqref{rgf-equality2}, we obtain that 
\begin{align*}
 \int_{\partial \Omega} u^{s}(x, \hat{y})  & \partial_{\nu (x)} \text{Im} \Phi (x,z) -  \partial_{\nu (x)} u^{s}(x, \hat{y}) \text{Im} \Phi (x,z) \, \text{d}s(x) \\ 
& =\int_{D} k^2 (n(x)-1) \text{Im} \Phi (x,z) u(x, \hat{y}) \, \text{d}x - \int_{\partial D} \mathrm{Im} \Phi (x,z) \mathscr{B}(u(x, \hat{y})) \, \mathrm{d}s(x) 
\end{align*}
which proves the claim.
\end{proof}

We know from \eqref{im-phi} that $\text{Im} \Phi(x,z)$ has maximal values whenever $x \approx z$ and decays as dist$(z,x) \rightarrow \infty$ for $d=2$ and 3. A consequence of this and the above result is that the imaging functional $W(z)$ is maximal in the closure of the scatterer $D$, where its values dissipate in the exterior of $D$. The following theorem formally states the decay rate of $W(z)$ as the sampling points $z$ tend away from $\overline{D}$.

\begin{theorem}\label{wz-decay}
The imaging functional $W(z)$ defined in \eqref{wz-raw} satisfies 
$$W(z) = \mathcal{O}\Big( \mathrm{dist}(z, \overline{D})^{\frac{1-d}{2}} \Big) \quad \text{as} \quad \mathrm{dist}(z,\overline{D}) \rightarrow \infty $$
for every $z \notin \overline{D}$ where $d=$ 2 or 3
\end{theorem}
\begin{proof}
Recall that $W(z)$ satisfies \eqref{wz-d}, which is a sum of a volume integral and a boundary integral on $D$ and $\partial D$, respectively. For the volume integral, we have that
\begin{align*} 
\bigg \rvert \int_{D} k^2 (n(x)-1) \text{Im} \Phi (x,z) &u(x, \hat{y}_j) \, \text{d}x  \bigg \rvert \\
&\leq C \norm{\text{Im}\Phi(\cdot , z)}_{L^{2}(D)} \norm{u (\cdot , \hat{y}_j)}_{{L}^{2}(D)} \\
&\leq C \norm{\text{Im}\Phi(\cdot , z)}_{L^{2}(D)}  \left\{ \| u^i  (\cdot , \hat{y}_j)\|_{{H}^{1}(D)} + \| u^i  (\cdot , \hat{y}_j) \|_{{H}^{1}(\partial D)} \right\}
\end{align*}
where we have used the definition of the total field and well--posedness of \eqref{bvp}--\eqref{2nd-bc}. Now, for the boundary integral we have that 
\begin{align*} 
\int_{\partial D} \text{Im} \Phi (x,z) & \mathscr{B}(u(x, \hat{y}_j)) \, \text{d}s(x) \\
&= \int_{\partial D}  \mu(x) \grad_{\partial D} \text{Im} \Phi (x,z) \cdot \grad_{\partial D} u(x, \hat{y}_j)  + \gamma(x) \text{Im} \Phi (x,z)  u(x, \hat{y}_j) \, \text{d}s(x).
\end{align*} 
Again, we can estimate 
\begin{align*} 
\bigg \rvert \int_{\partial D} \text{Im} \Phi (x,z) & \mathscr{B}(u(x, \hat{y}_j)) \, \text{d}s(x)   \bigg \rvert \\
&\leq C \norm{\text{Im}\Phi(\cdot , z)}_{H^{1}(\partial D)} \norm{u (\cdot , \hat{y}_j)}_{H^{1}(\partial D)} \\
&\leq C \norm{\text{Im}\Phi(\cdot , z)}_{H^{1}(\partial D)}  \left\{ \| u^i  (\cdot , \hat{y}_j)\|_{{H}^{1}(D)} + \| u^i  (\cdot , \hat{y}_j) \|_{{H}^{1}(\partial D)} \right\}.
\end{align*}
Note that we have again used the well--posedness of \eqref{bvp}--\eqref{2nd-bc}. 

This gives that the imaging functional satisfies 
$$W(z) \leq C \left\{  \norm{\text{Im}\Phi(\cdot , z)}_{L^{2}( D)} +  \norm{\text{Im}\Phi(\cdot , z)}_{H^{1}(\partial D)} \right\}.$$
We appeal to the asymptotic behavior of $\text{Im}\Phi(x , z)$, as defined is \eqref{im-phi} 
$$ \text{Im}\Phi(x , z) = \mathcal{O} \Big ( |x-z|^{\frac{1-d}{2}} \Big ) \quad \text{as} \quad |x-z| \rightarrow \infty .$$
We also appeal to the estimates
$$ |\nabla \text{Im} \Phi (x,z)|  \leq k |J_{1}(k|x-z|)| \,\, \text{ for $d=2$, and }\,\,   |\nabla \text{Im} \Phi (x,z)|  \leq k |j_{1}(k|x-z|)| \,\, \text{ for $d=3$}$$
where $J_1$ and $j_1$ are the first-order Bessel function and spherical Bessel function of the first kind. We have that
$$J_{1}(k|x-z|) = \mathcal{O} \Big (|x-z|^{-\frac{1}{2}} \Big ) \quad \text{and } \quad j_{1}(k|x-z|) = \mathcal{O} \Big (|x-z|^{-1} \Big )$$ 
as $ |x-z| \rightarrow \infty$, which proves the claim by appealing to the definition of the surface gradient. 
\end{proof}

The above result implies that the imaging functional $W(z)$ decays as the sampling point moves away from the scatterer. Since $W(z)$ can be computed from the measured data we have a simple numerical method for recovering the scatterer $D$. By Theorem 2.3 in \cite{lan-nguyen} we have that our imaging functional is stable with respect to noisy data.

Recall, that by the radiation condition, we have that $\partial_{\nu} u^s \approx \text{i}ku^s$ when the boundary measurements are taken on a large-radius ball. In this case, the modified imaging functional is defined as
\begin{equation}\label{w_far}
W_{\text{far}} (z) := \sum_{j=1}^{J} \bigg \rvert \int_{\partial \Omega} u^s (x , \hat{y}_j) \partial_{\nu (x)} \text{Im}\Phi(x,z) - \text{i}k u^s (x , \hat{y}_j) \text{Im}\Phi(x,z) \, \text{d}s(x) \bigg \rvert .
\end{equation}
This modified imaging functional approximates $W(z)$ and only depends on the measured data $u^{s}(x,\hat{y})$.

\section{Numerical Validation for the Inverse Scattering Problem}\label{numerical-section}
We provide numerical examples for the reconstruction of the scatterer $D$ using the imaging functional $W(z)$ as defined in \eqref{wz-raw}. We cover two distinct cases which are determined in the way in which the data, i.e. the scattered field $u^s (x,\hat{y})$, was computed on the boundary of $\Omega$. The first case uses separation of variables as well as the Lippmann-Schwinger integral equation to recover circular regions. For all cases, the measurements boundary $\partial \Omega$ is given by a circle centered at the origin. In the second case, we derive a system of boundary integral equations to compute the synthetic data where the scatterer $D$ is not necessarily circular.  Also, for most numerical examples in this section, we plot the normalized values of $W$. That is, we plot
$$W_{\text{nor}} (z) := \Bigg( \frac{W(z)}{\norm{W(z)}_{\infty}} \Bigg )^\rho$$ 
where $\rho>0$ is a parameter used for increasing the resolution of the reconstructions. In order to compute $W_{\text{nor}}(z)$, we approximate the boundary integral on $\partial \Omega$ with an 32--point Riemann sum.

Thus, we compute $W_{\text{nor}}(z)$ using synthetic data. See the Appendix for details on how the synthetic data is computed. 
We approximate the representation of this Cauchy data as $32 \times 32$ matrices given by
$${\bf{us}}= \big [ u^s ({x}_i , \hat{y}_j)\big ]^{32}_{i,j=1} \quad \text{and} \quad {\bf{Dus}} = \big [  \partial_{r} u^s ({x}_i , \hat{y}_j)\big ]^{32}_{i,j=1}$$
where ${x}_i =\hat{y}_i = \big( \text{cos}(\theta_i), \text{sin}(\theta_i)\big)$ where $\theta_i = 2\pi (i - 1)/32$ for $i = 1,\dots, 32$. We add random noise to this synthetic data, where we define
$$ {\bf{us}}_{\delta}= \big [ \textbf{us}(i,j) (1 + \delta \textbf{E}^{(1)}(i,j) )\big ]^{32}_{i,j=1} \quad \text{and} \quad {\bf{Dus}}_{\delta} = \big [ \textbf{Dus}(i,j) (1 + \delta \textbf{E}^{(2)}(i,j) )\big ]^{32}_{i,j=1},$$
where$ \| \textbf{E}^{(j)}\|_2 = 1$ for $j=1,2$. 
The matrices $\textbf{E}^{(j)} \in \mathbb{C}^{32 \times 32}$ have entries with real and imaginary parts uniformly distributed in $[-1,1]$, then normalized so that $\delta \in (0,1)$ represents the relative noise level added to the data. For our reconstructions, we sample from a uniform $150 \times 150$ grid point on the rectangular region $[-1, 1] \times [-1, 1]$ to recover the scatterer. {\color{black}In all our examples, we take the parameter $\rho=4$ to sharpen the resolution.}\\

\noindent{\bf Example 1a:} For the first example in Figure \ref{fig0}, we suppose that the scatterer $D = B(0,0.5)$. We assume that the refractive index is $n = 5$ and boundary parameters $\mu = 1.5$ and  $\gamma = 2$. In both subplots, we fix $\delta = 0.05$, which corresponds to $5 \%$ random noise.
\begin{figure}[H]
\centering 
\includegraphics[scale=0.15]{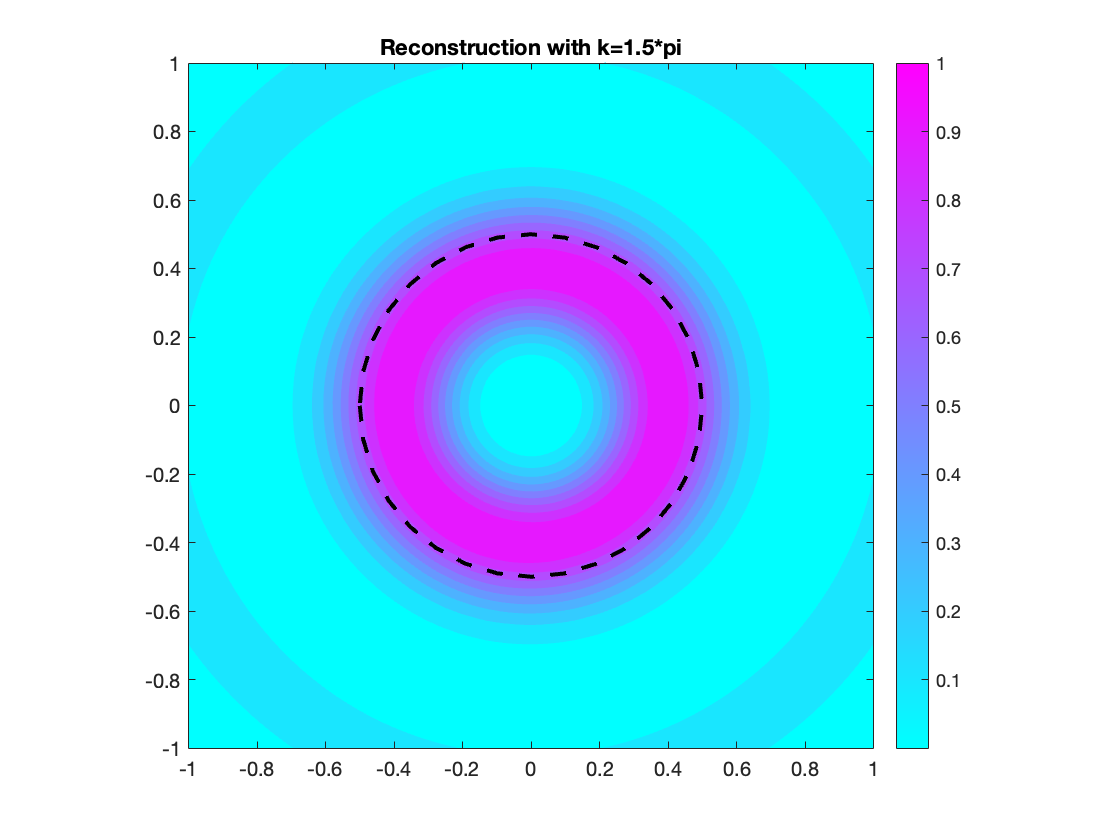}
\includegraphics[scale=0.15]{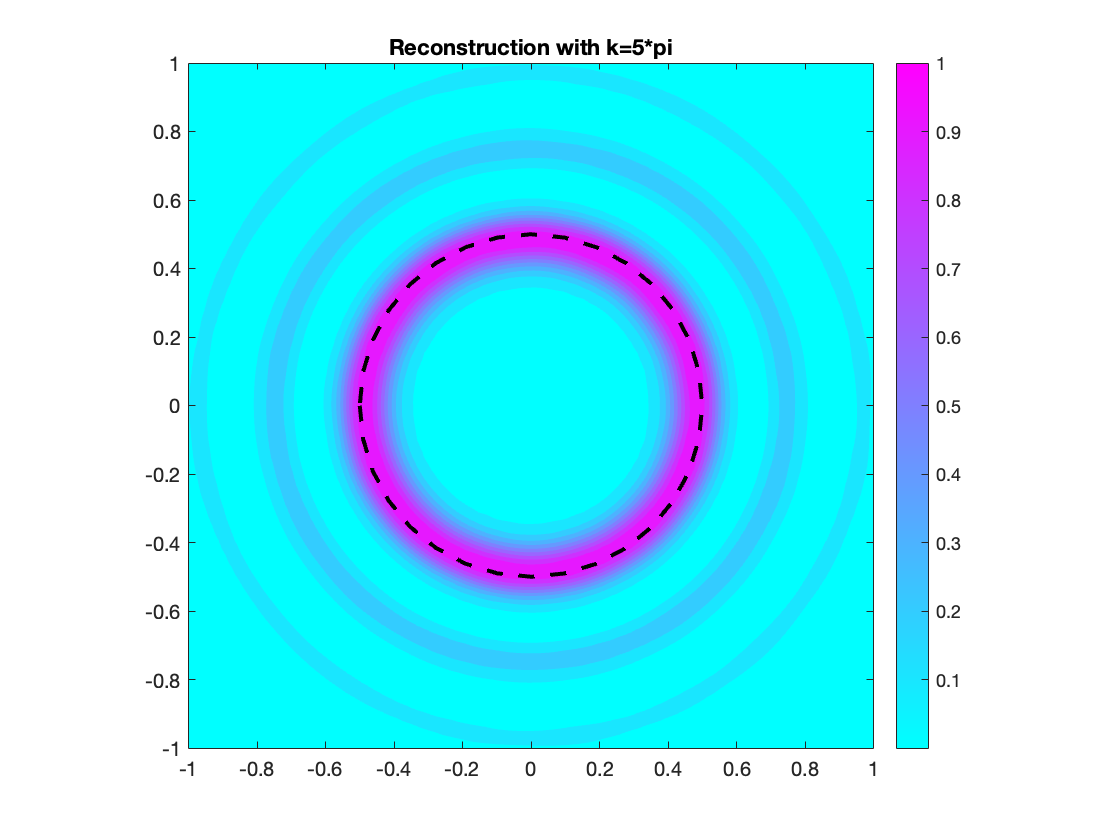}
\caption{Reconstruction of a circular region with radius $R=0.5$. On the left: wavenumber $k = 3\pi/2$. On the right: wavenumber $k = 5\pi$.}
\label{fig0}
\end{figure}

\noindent{\bf Example 1b: } Noticing that a higher wave number produces a sharper image, we want to demonstrate the stability of $W_{\text{nor}}(z)$ under random noise. Thus, we adopt the same settings as the previous sub example. That is, we suppose that the scatterer $D = B(0,0.5)$. Again, we assume that the refractive index is $n = 5$ and boundary parameters $\mu = 1.5$ and  $\gamma = 2$. Now, we fix the wavenumber to be $k = 5 \pi$ in Figure \ref{fig1}. 

\begin{figure}[H]
\centering 
\includegraphics[scale=0.15]{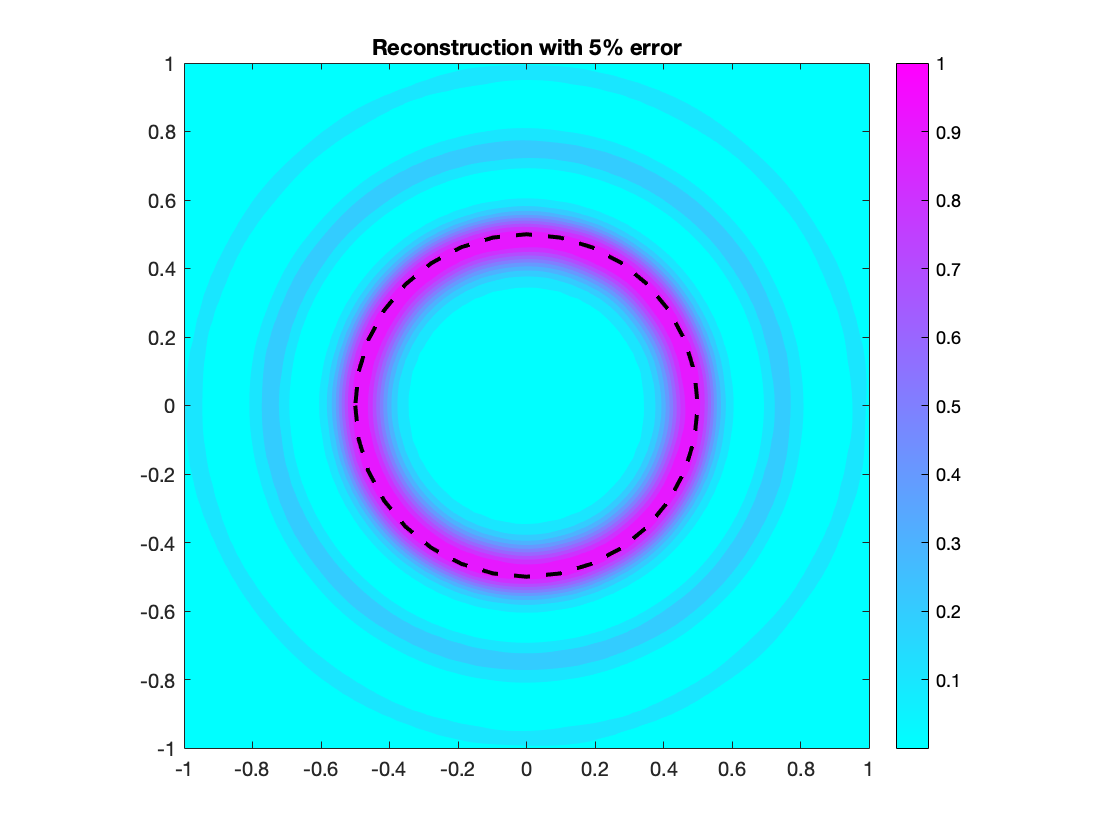}
\includegraphics[scale=0.15]{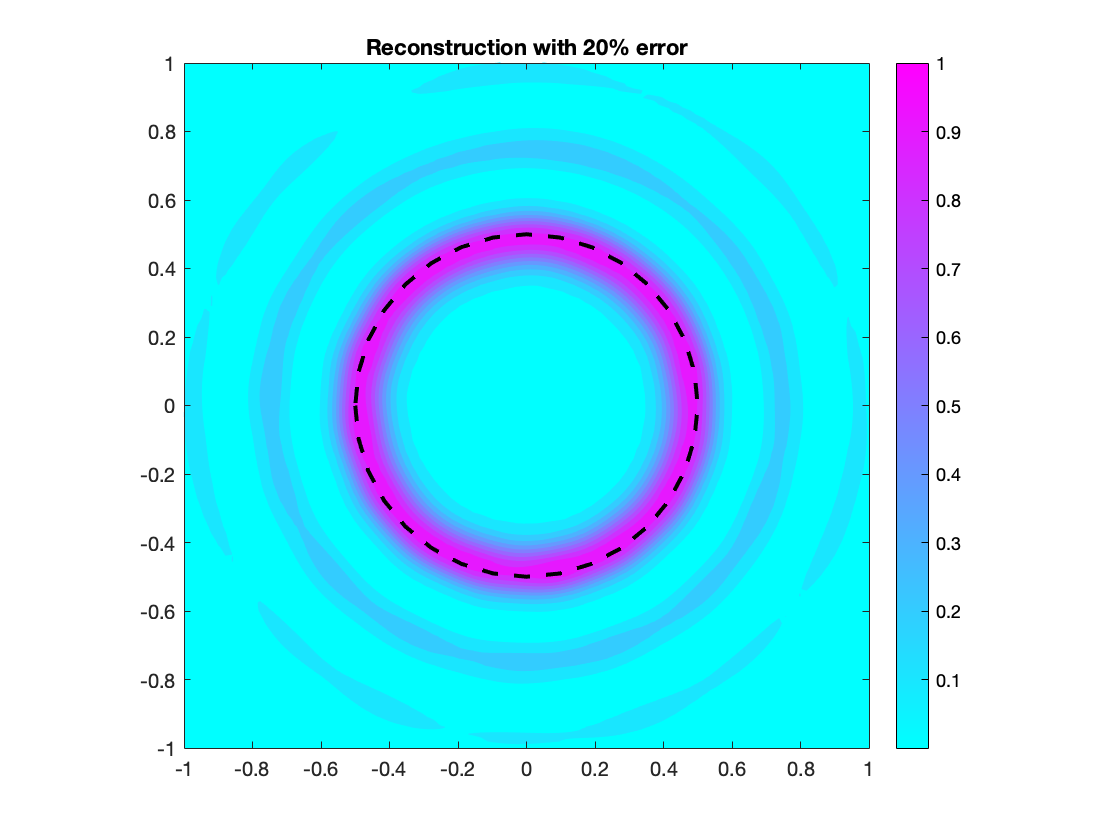}
\caption{Reconstruction of a circular region with radius $R=0.5$. On the left: $5\%$ error was added to the Cauchy data. On the right: $20\%$ error was added to the Cauchy data.
}
\label{fig1}
\end{figure}

\noindent
\textbf{Numerical Reconstruction of Small Regions:}\\
Here, we consider the case where the scatterer $D$ is a small volume region. More precisely,
$$
D = \bigcup_{j=1}^{J} D_j \quad \text{with} \quad D_j = (x_j + B(0,0.1)) \quad \text{such that} \quad \text{dist}(x_i , x_j) \geq c_0>0
$$
for $ i \neq j$, where $B(0,0.1)$ is the circle centered at the origin of radius 0.1. We also assume that the individual regions $D_j$ are disjoint.
Here we use the Lippmann-Schwinger integral equation to compute the measured data. By using Green's second identity just as in the previous section we have that the scattered field is given by 
$$u^s(x,\hat{y}) = \int_{D} k^2 \big( n(\omega)-1 \big) \Phi (x,w) u(\omega,\hat{y}) \, \text{d}\omega - \int_{\partial D} \Phi (x,\omega) \mathscr{B}\big(u(\omega,\hat{y})\big) \, \text{d}s(\omega).$$
Since the volume of $D$ is relatively small, we can approximate the scattered field as
$$u^s(x,\hat{y}) \approx \int_{D} k^2 \big( n(\omega)-1 \big) \Phi (x,\omega) u^i(\omega,\hat{y}) \, \text{d}\omega - \int_{\partial D} \Phi (x,\omega) \mathscr{B}\big(u^i(\omega,\hat{y})\big) \, \text{d}s(\omega). $$
Once again, $\Phi (x, w)$ is the free-space Green's function for the Helmholtz equation. The volume integral is computed via Gaussian quadratures, and the boundary integral uses the `integral' command in MATLAB. The normal derivative $\partial_r u^s(x,\hat{y})$ is similarly computed by taking the derivative of the above expression.\\

\noindent{\bf Example 2: } Here, we once again suppose that $\partial \Omega = B(0,1)$. In this example, we consider cases where $J=2$ and $J=3$ in Figure \ref{fig2}. For the case when $J=2$, we assume that the centers are $x_1 = (-0.5, -0.5)$ and $x_2 = (0.5, 0.5)$. For the case when $J=3$, we assume that the centers are $x_1 = (0.35, 0.65)$, $x_2 = (-0.6, 0.1)$, and $x_3 = (-0.2, -0.7)$. In either cases, we assume that the refractive index is $n = 1+2 \text{i}$, where the boundary parameters $\mu = 2.5 - 1.5 \text{i}$ and  $\gamma = 1.5 - 2\text{i}$. We take the wavenumber to be $k = 3 \pi$. For both cases, we plot the imaging functional with random noise level $\delta = 0.05$ which corresponds to $5 \%$ to simulated data $u^s$ and $\partial_{r} u^s$ on $\partial \Omega$. \\

\begin{figure}[h!]
\centering 
\includegraphics[scale=0.17]{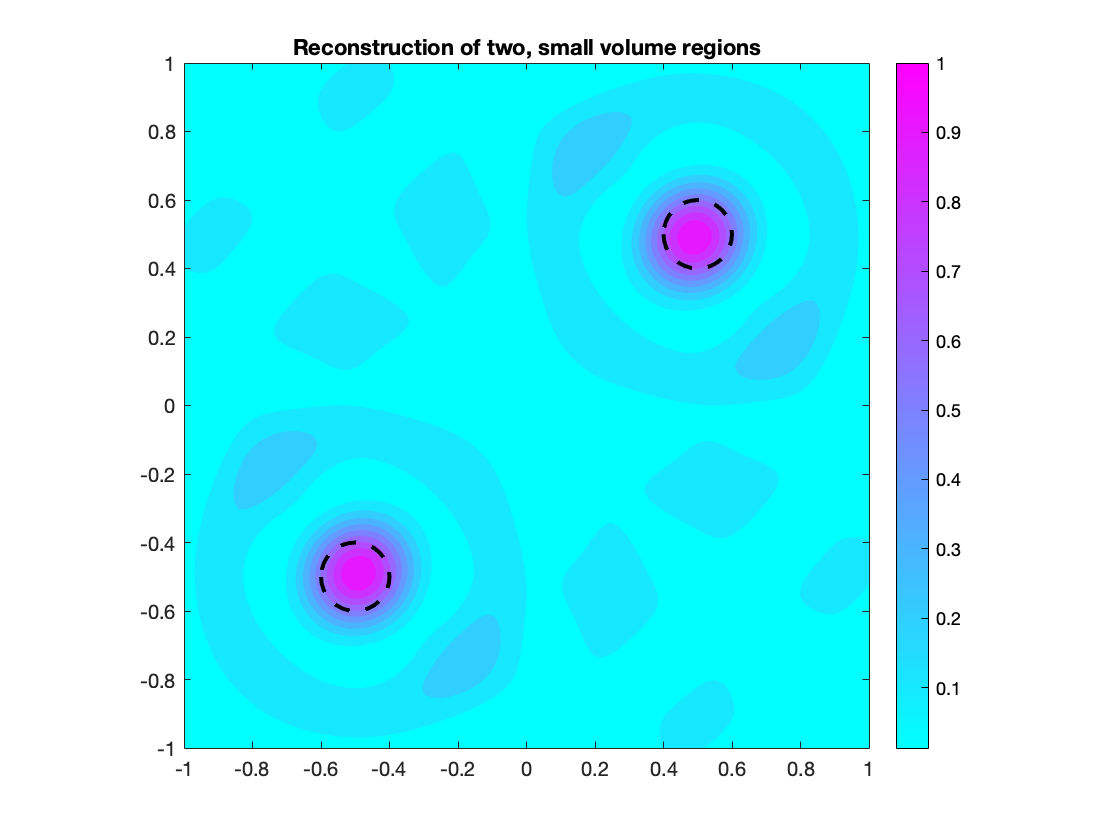}
\includegraphics[scale=0.17]{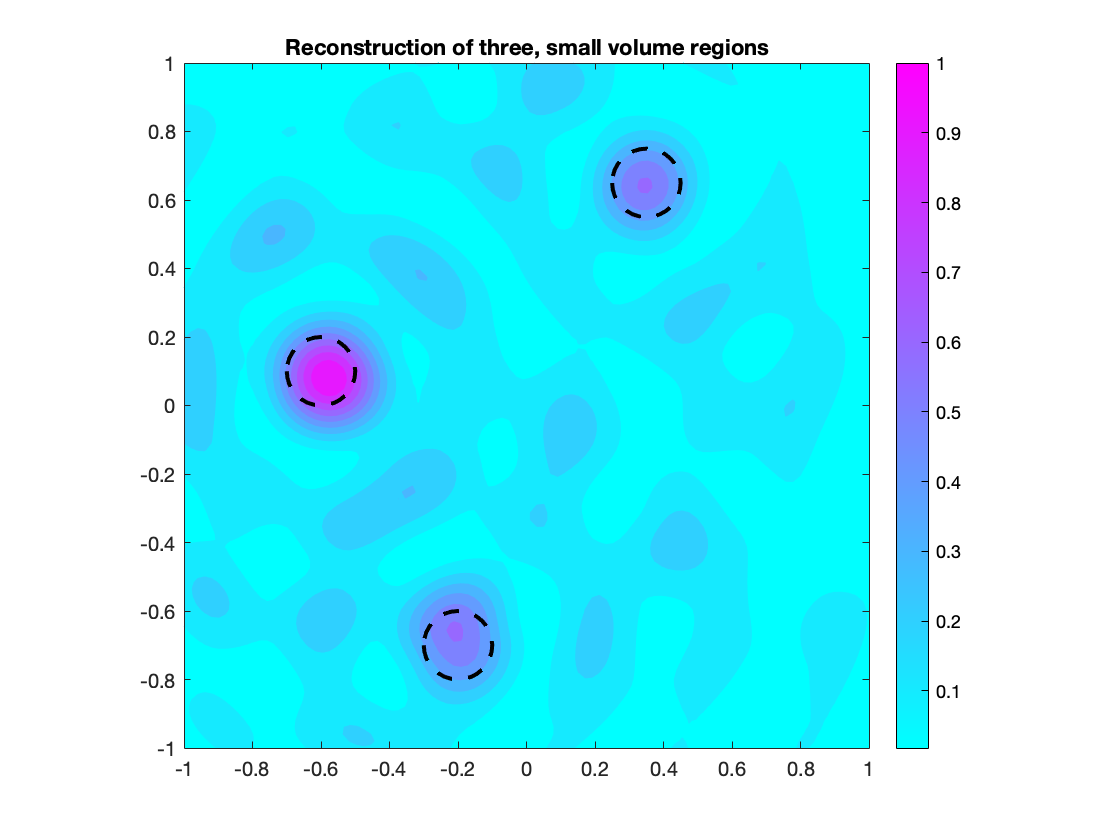}
\caption{Reconstruction of small circular regions. On the left: two circular regions with centers at $(-0.5, -0.5)$ and $(0.5, 0.5)$. On the right: three circular regions with centers at  $(0.35, 0.65)$, $(-0.6, 0.1)$, and $(-0.2, -0.7)$.}
\label{fig2}
\end{figure}

\noindent
\textbf{Numerical Reconstructions With Far Away Measurements:}\\
We consider the case when the measurements boundary $\partial \Omega$ is far away from the scatterer. Therefore, we will assume $D \subset [-1,1] \times [-1,1]$, where 
$$ \partial \Omega=10 (\text{cos}(\theta) , \text{sin}(\theta)) $$
i.e. $ \Omega = B(0,10)$. By the radiation condition, it is known that $\partial_{r}u^s  \approx \text{i}k u^s$ when the measurements boundary are far away from the scatterer. For this case, we use $W_{\text{far}}(z)$ as defined in \eqref{w_far}, instead of $W(z)$. Notice that this means that one only need to compute/measure $u^s$ on $\partial \Omega$. That is, we plot
$$ \widetilde{W}_{\text{nor}}(z) := \bigg ( \frac{W_{\text{far}}(z)}{\| W_{\text{far}}(z)\| _{\infty}} \bigg)^\rho ,$$ 
where again $\rho=4$ is used for increasing the resolution of the reconstructions. \\

\noindent
\textbf{Example 3:} Here we provide two examples in Figure \ref{fig3}. For the first example, we compute the data via separation of variables to recover $D = B(0,0.4)$. The refraction index $n=2$ with boundary parameters $\mu = 0.8$ and $\gamma = 1.1$. In the second example, we use the Lippmann-Schwinger integral equation to recover two circular regions of radius 0.1 centered at $(-0.4 , -0.4)$ and $(0.4 , 0.4)$. The refraction index is $n = 4+\text{i}$ where the boundary parameters are set to $\mu = 0.9 - 2\text{i}$ and $\gamma = 3-\text{i}$. For both plots, we add $5\%$ error to the synthetic data. 

\begin{figure}[h!]
\centering 
\includegraphics[scale=0.17]{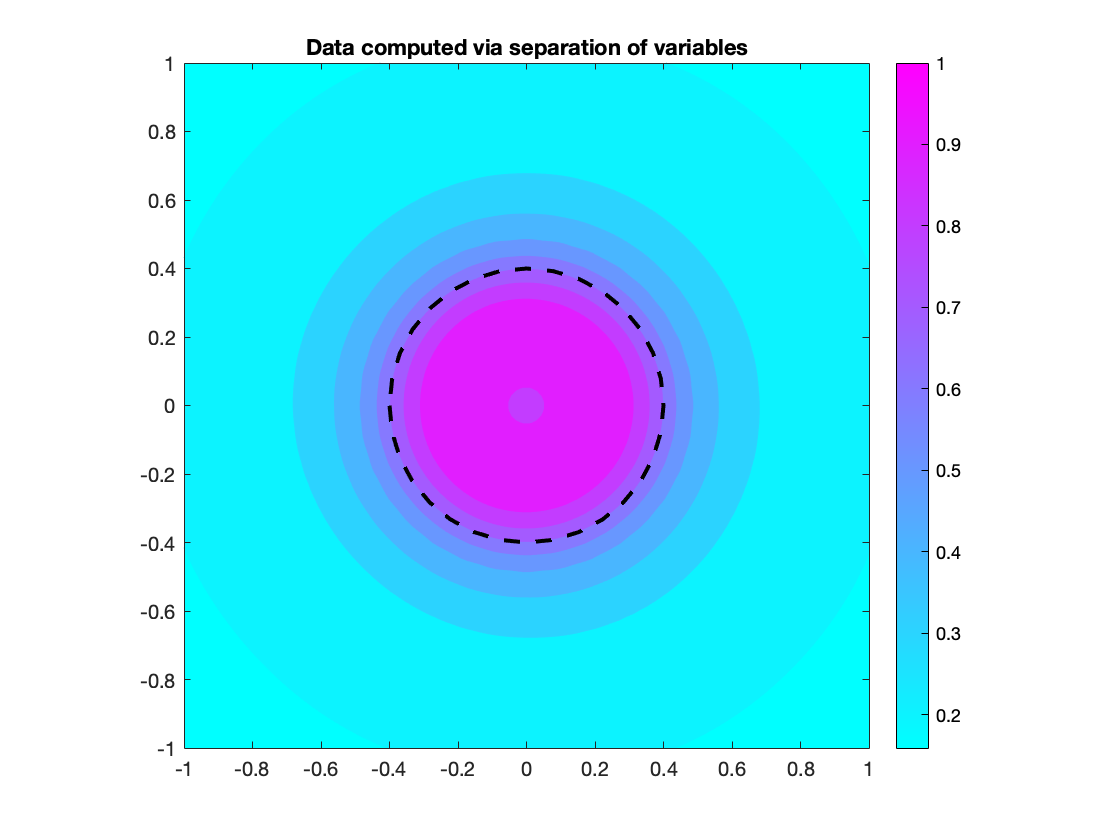}
\includegraphics[scale=0.17]{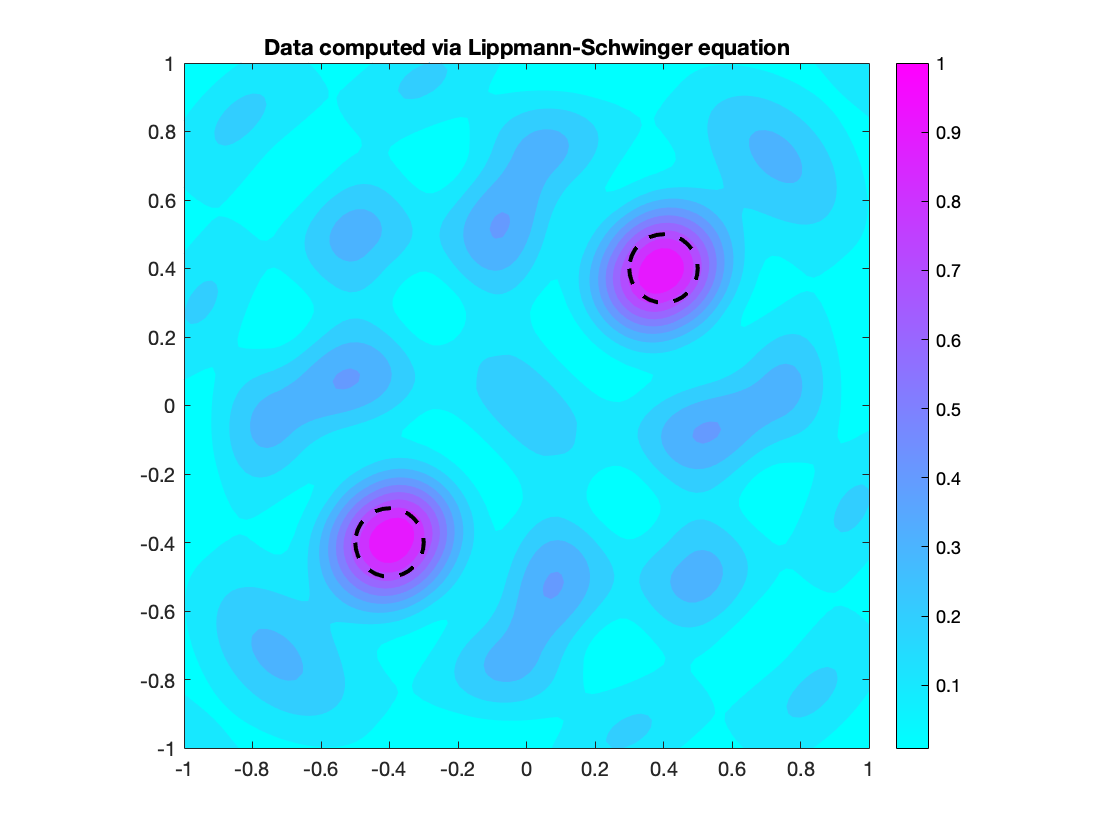}
\caption{Reconstructions where we approximate $\partial_{r}u^s \approx \text{i}k u^s$. Contour plots of $\widetilde{W}_{\text{nor}}(z)$ with $5\%$ error. On the left: data computed via separation of variables to recover $D = B(0,0.4)$ with $k=2\pi$. On the right: data computed via Lippmann-Schwinger equation to recover circular regions or radius 0.1 centered at  $(-0.4 , -0.4)$ and $(0.4 , 0.4)$ with $k=2.5\pi$.}
\label{fig3}
\end{figure}

{\color{black}Now, in order to provide more examples we will solve the direct scattering problem using a system of boundary integral equations(BIEs) to compute the synthetic Cauchy data. We can then proceed as described above to compute the imaging functional.} With this, we provide numerical reconstructions for the circle, elliptical, and kite-shape scatterer. For these examples, we fix the measurement boundary 
$$ \partial \Omega=3 (\text{cos}(\theta) , \text{sin}(\theta)) $$
with 64 equally spaced incident direction and observations. Furthermore, we also fix the wave number $k = 3\pi/2$, refractive index $n=5$, and boundary parameters $\gamma = 2$, $\mu = 1.5$.\\

\noindent
\textbf{Example 6:} For comparison, we begin by reconstructing a circular region, where the scatterer $D = B (0,0.5)$. We provide the cases where we add $5\%$ and $10\%$ error to the Cauchy data.

\begin{figure}[H]
\centering 
\includegraphics[scale=0.16]{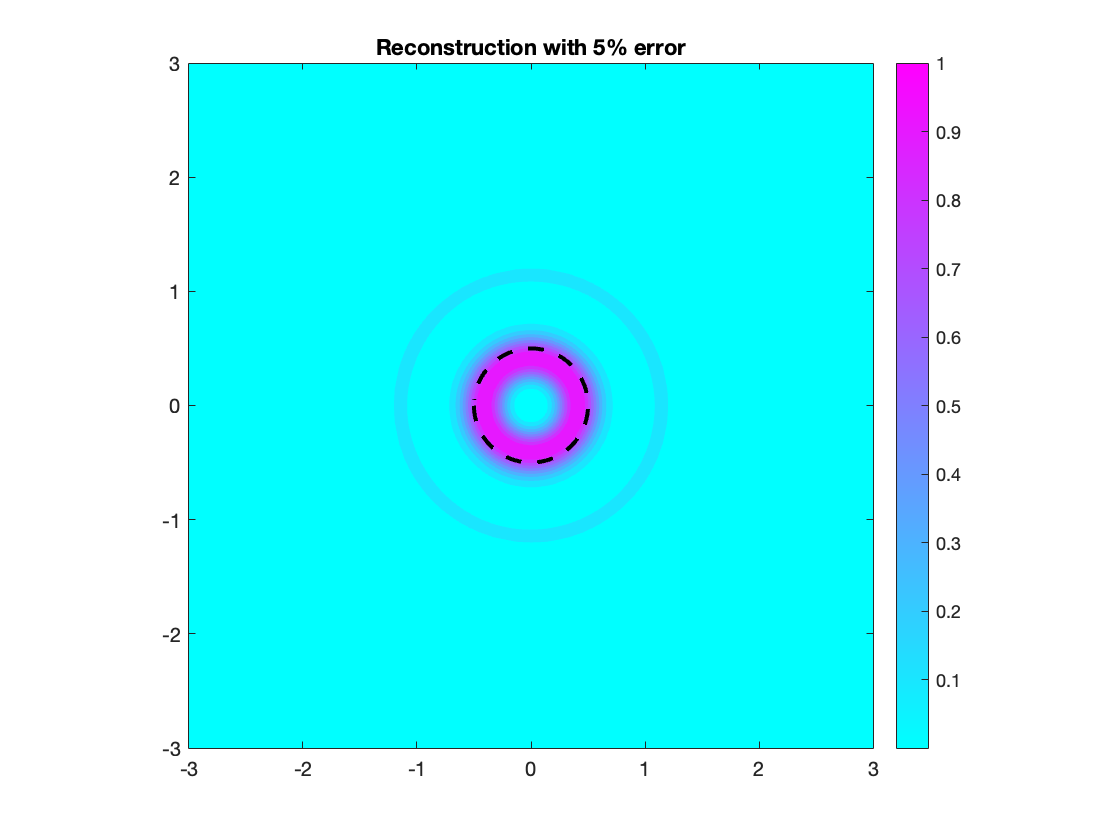}
\includegraphics[scale=0.16]{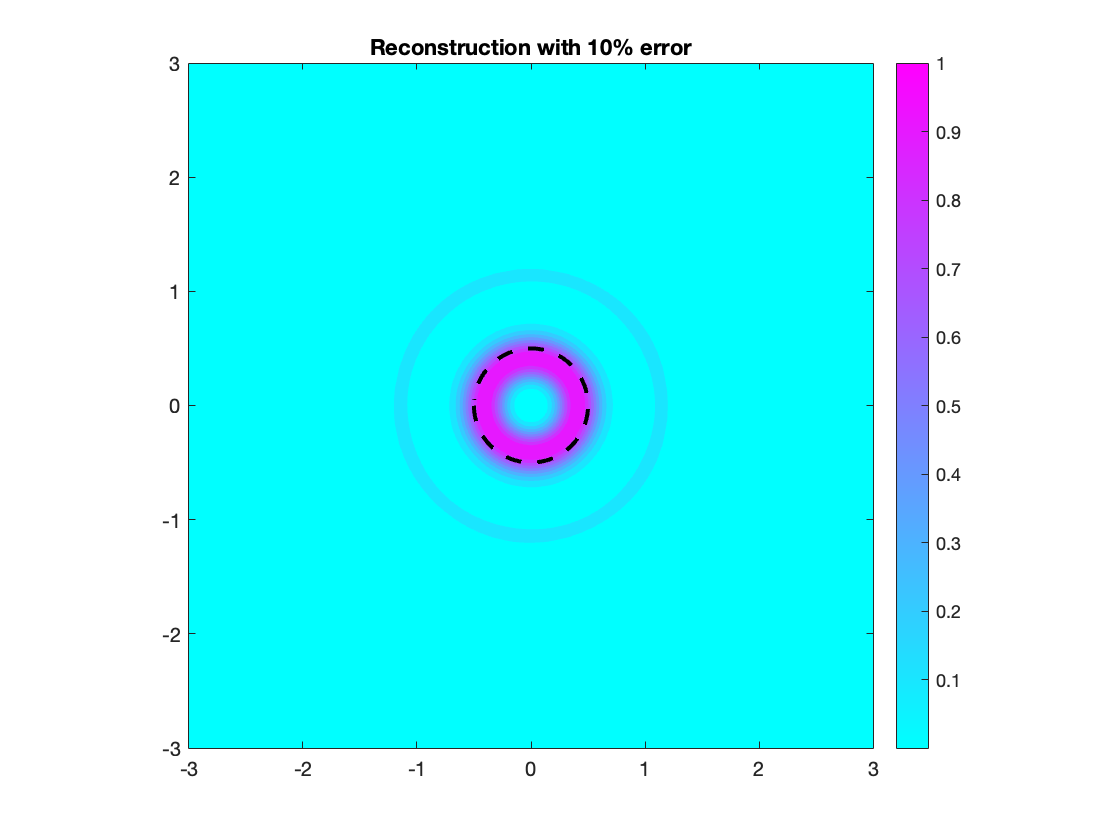}
\caption{Reconstruction of the circle shape where the Cauchy data was computed via BIEs. On the left: $5\%$ error was added to the data. On the right: $10 \%$ error was added to the data.}
\label{fig6}
\end{figure}

\noindent
\textbf{Example 7:} Here we recover an ellipse-shaped scatterer $D$, where the boundary is parameterized in polar coordinates as
$$\partial D = \big(\text{cos}(\theta) , 0.9\cdot\text{sin}(\theta) \big ), \quad \theta \in [0, 2\pi ).$$
Once again, we plot the cases where we add $5\%$ and $10\%$ error to the Cauchy data.\\

\begin{figure}[H]
\centering 
\includegraphics[scale=0.16]{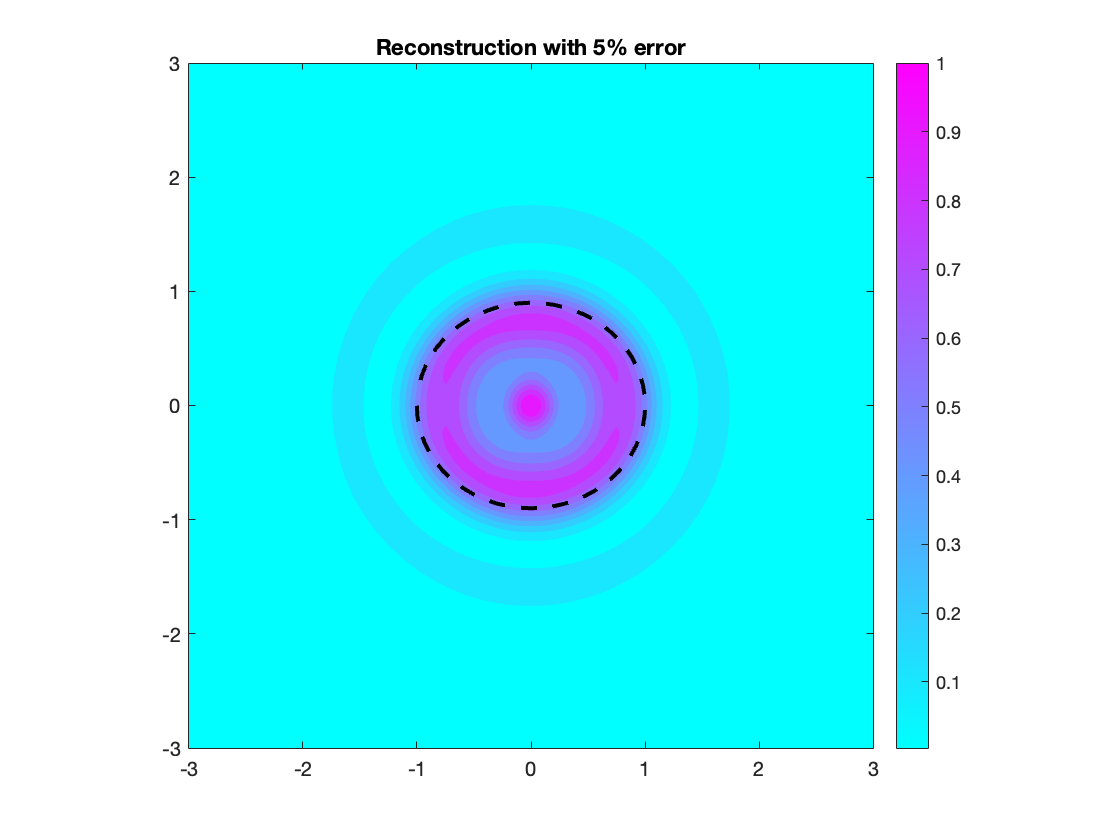}
\includegraphics[scale=0.16]{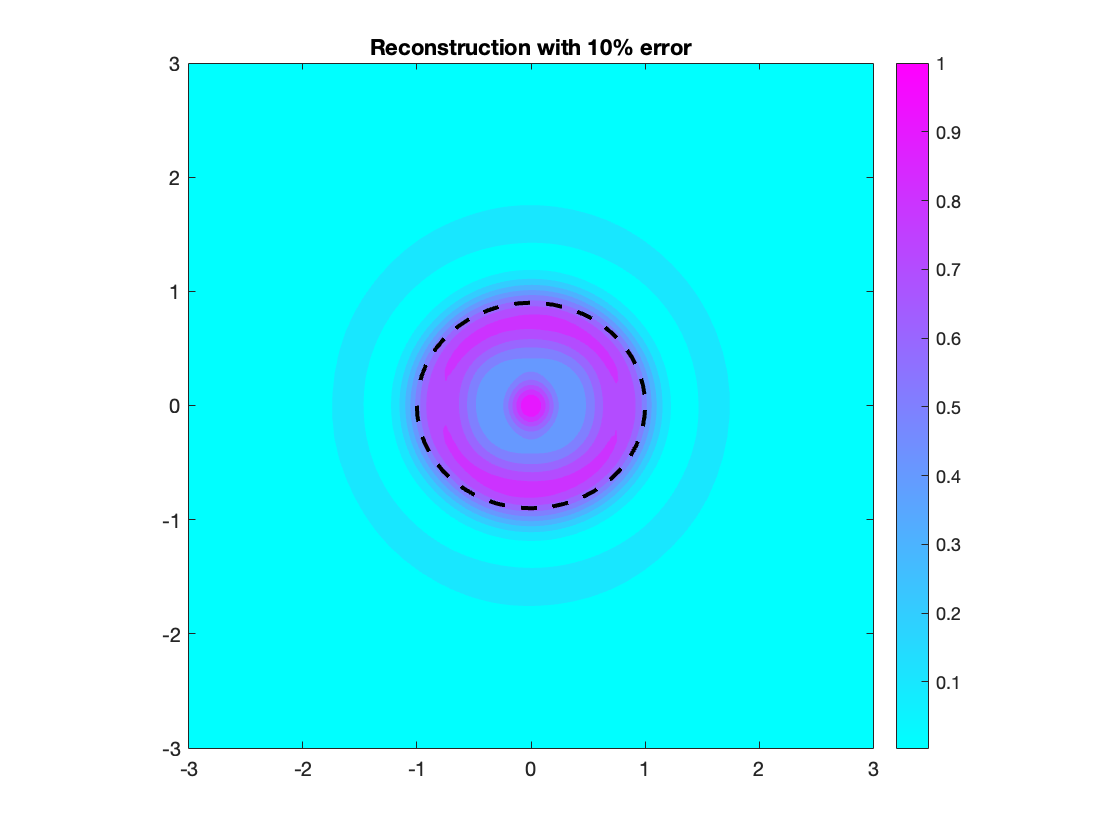}
\caption{Reconstruction of the ellipse shape where the Cauchy data was computed via BIEs. On the left: $5\%$ error was added to the data. On the right: $10 \%$ error was added to the data.}
\label{fig7}
\end{figure}

\noindent
\textbf{Example 8:}  We now provide a reconstruction of a kite-shaped scatterer $D$, where the boundary is parameterized in polar coordinates as
$$\partial D= \big(-1.5\cdot \text{sin}(\theta), \text{cos}(\theta)+0.65\cdot \text{cos}(2\theta)-0.65 \big), \quad \theta \in [0, 2\pi ).$$
As in our previous examples, we plot the cases where we add $5\%$ and $10\%$ error to the Cauchy data.
\begin{figure}[H]
\centering 
\includegraphics[scale=0.16]{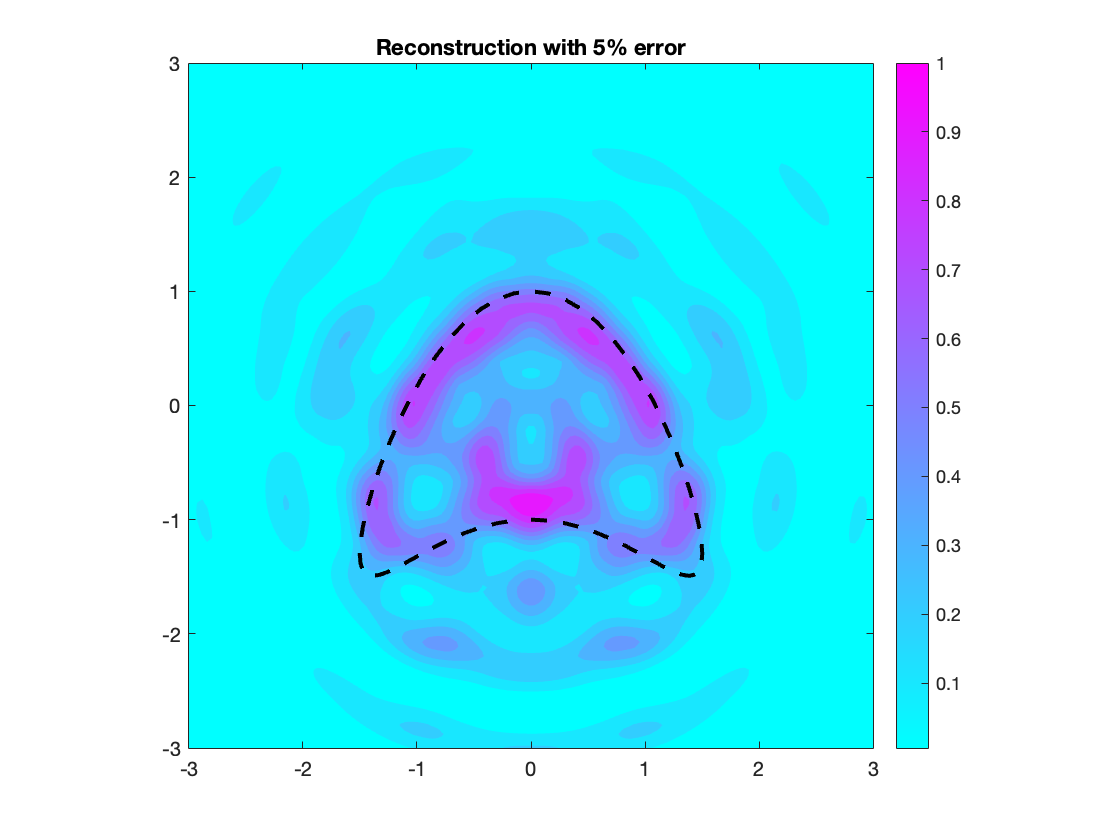}
\includegraphics[scale=0.16]{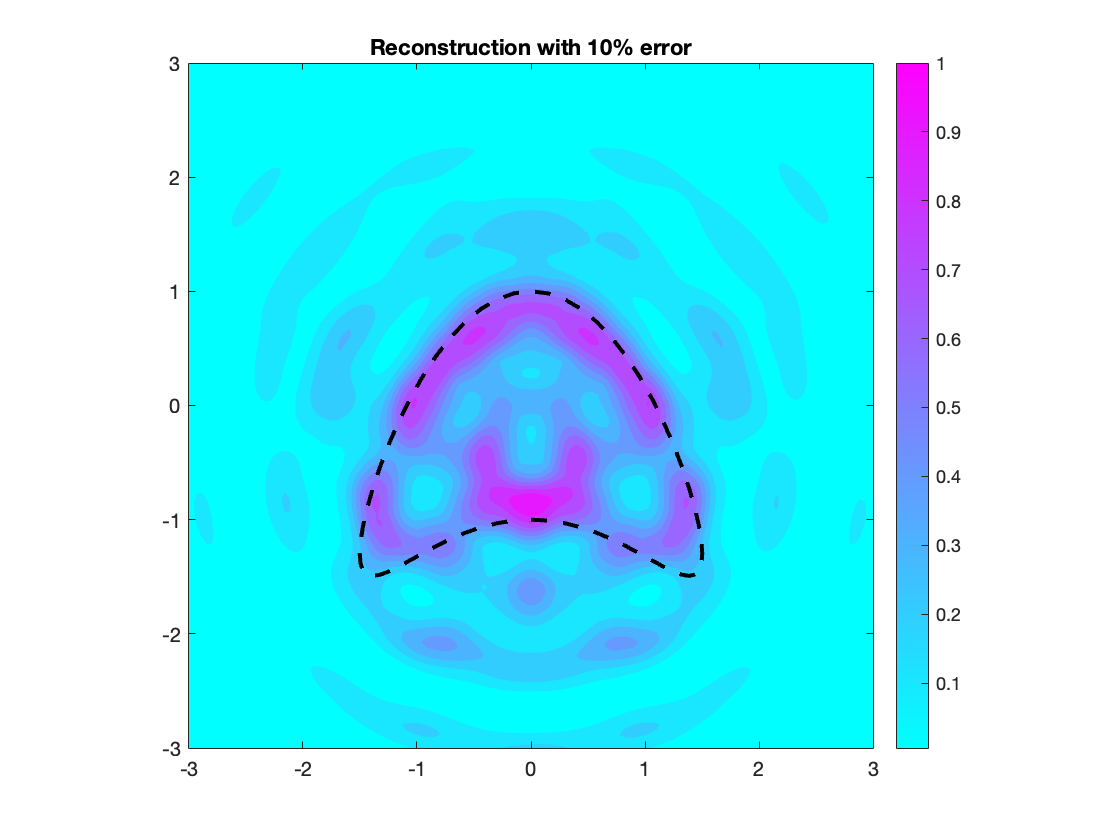}
\caption{Reconstruction of the kite shape where the Cauchy data was computed via BIEs. On the left: $5\%$ error was added to the data. On the right: $10 \%$ error was added to the data.}
\label{fig8}
\end{figure}

From our examples in this section, we provided reconstructions for circular and non-circular regions under sign assumptions on the refraction index \eqref{refraction-index} and boundary coefficients \eqref{gamma-bounds} and \eqref{mu-bounds}. Throughout our numerical examples, the imaging functional $W(z)$ demonstrated stability with respect to noisy data. Thus, our sampling method provides high fidelity reconstructions for isotropic scatterers with a delaminated boundary.

\section{The Associated Transmission Eigenvalue Problem}\label{discrete-TEV}
In this section, we are interested in the associated transmission eigenvalue problem. This problem is derived by assuming that there is an incident field (solving the Helmholtz equation) that does not produce a scattered field outside the scatterer $D$ for  \eqref{bvp}--\eqref{2nd-bc}. {\color{black} Notice, that the analysis of our imaging functional in the previous section requires the measured Cauchy data. The transmission eigenvalue problem we wish to study here is derived when the measured Cauchy data$=(0,0)$. Therefore, we wish to study this case as to complete the analysis of our reconstruction method.}

We will prove that the set of transmission eigenvalues is at most discrete for the case where the refraction index Re$(n)>1$ in the scatterer $D$. In general, sampling methods to recover a scatterer are not valid when the wave number $k$ is a transmission eigenvalue (see e.g. \cite{guo}). We analyze the case where an incident field satisfying the Helmholtz equation in $\mathbb{R}^d$ does not produce a scattered field on the exterior of $D$. Therefore, we have for $w = u^s + u^i$ and $v = u^i$ in $\widetilde{H}^{1}(D)$ satisfies
\begin{align}
\Delta v + k^2 v = 0 \quad &\text{and} \quad \Delta w + k^2 n w = 0 \quad \text{in} \enspace D,\label{v-w-in-D}\\
w = v \quad &\text{and} \quad \partial_{\nu} v - \partial_{\nu}w= \mathscr{B}(w) \quad \text{on} \enspace \partial D. \label{v-w-on-bD}
\end{align}
Note we have used the assumption that $u^s = 0$ in $\R^d \setminus \overline{D}$. We now define 
$$\widetilde{H}^{1}(D) = \brac{\varphi \in H^{1}(D) \enspace \text{such that} \enspace \varphi \rvert_{\partial D} \in H^{1}(\partial D)}.$$ 
The values $k \in \C$ for which \eqref{v-w-in-D}--\eqref{v-w-on-bD} has a non-trivial solution $(w,v) \in  \widetilde{H}^{1}(D) \times \widetilde{H}^{1}(D)$ are called the transmission eigenvalues (TEs). 

We begin by considering the case where Re$(n)>1$. Multiplying the second equation by a test function $\phi \in \widetilde{H}^{1}(D)$ and by applying Green's 1st Theorem, we have that
$$ 0 = \int_{D} \nabla w \cdot \nabla \overline{\phi} - k^2 n w \overline{\phi} \, \text{d}x - \int_{\partial D} \overline{\phi} \partial_{\nu} w \, \text{d}s .
$$

\noindent
By the boundary condition \eqref{v-w-on-bD}, we get that
$$
0 = \int_{D} \nabla w \cdot \nabla \overline{\phi} - k^2 n w \overline{\phi} \, \text{d}x - \int_{\partial D} \overline{\phi} \big( T_k w - \mathscr{B}(w) \big)  \, \text{d}s .
$$
Here the interior DtN map for the Helmholtz equation is given by
$$T_{k} : H^{1/2}(\partial D) \rightarrow  H^{-1/2}(\partial D) \quad \text{defined by} \quad T_{k} f = \partial_{\nu}v$$
where
$$\Delta v + k^2 v = 0 \enspace \text{in} \enspace D \quad \text{and} \quad v = f \enspace \text{on} \enspace \partial D.$$
Note that $T_k$ is a well--defined bounded linear operator whenever $k^2$ is not a Dirichlet eigenvalue of $-\Delta$ in $D$. Also, $T_k$ depends analytically on $k$ for all values for which it is well--defined (see for e.g. \cite{Hsiao}). This now inspires us to define the bounded sesquilinear form $a_k (\cdot , \cdot) : \widetilde{H}^{1}(D) \times \widetilde{H}^{1}(D) \rightarrow \mathbb{C}$ as
\begin{equation} \label{a-k}
a_k (w, \phi ) := \int_{D} \nabla w \cdot \nabla \overline{\phi} - k^2 n w \overline{\phi} \, \text{d}x - \int_{\partial D} \overline{\phi} \big( T_k w - \mathscr{B}(w) \big) \, \text{d}s .
\end{equation}
Thus, $k$ is a TE of \eqref{v-w-in-D}--\eqref{v-w-on-bD} if and only if $a_k (w , \phi) = 0$ has a non-trivial solution $w \in \widetilde{H}^{1}(D)$ for all $\phi \in \widetilde{H}^{1}(D)$, provided that $k^2$ is not a Dirichlet eigenvalue of $- \Delta$ in $D$. We denote these Dirichlet eigenvalues as $\brac{\lambda_{j}(D)}_{j=1}^{\infty}$. 

We take a similar approach and analysis to \cite{Hughes} in the sense that we will prove discreteness of the TE by appealing to the Analytic Fredholm Theorem. We will show that the sesquilinear form $a_k (\cdot , \cdot)$ associated with the TE problem is represented by a Fredholm operator with index zero that depends on the wave number $k^2 \in \C \setminus \brac{ {\lambda_{j}(D)}}_{j=1}^{\infty}$ analytically. Consider the the following decomposition $a_k(\cdot , \cdot ) = b(\cdot , \cdot ) + c_k (\cdot , \cdot )$ where we define
\begin{equation}\label{b-ses}
b(w , \phi ) = \int_{D} \nabla w \cdot \nabla \overline{\phi} +  w \overline{\phi} \, \text{d}x - \int_{\partial D} (T_0 w - \mathscr{B}(w)) \overline{\phi} \, \text{d}s. 
\end{equation}
Here $T_0$ is the interior DtN mapping with the wave number $k = 0$ and we let
\begin{equation}\label{c-ses}
c_k (w , \phi) := - \int_{D} (k^2n +1) w \overline{\phi} \, \text{d}x - \int_{\partial D}  \overline{\phi}(T_k - T_0)w \, \text{d}s .
\end{equation}
The following lemmas regarding analytical properties of $a_k$, $b$, and $c_k$ will ultimately allow us to invoke the Analytic Fredholm Theorem.

\begin{lemma}
Let $b(\cdot , \cdot)$ and $c_k (\cdot , \cdot)$ be defined as in \eqref{b-ses} and \eqref{c-ses}, respectively. Then, $b(\cdot , \cdot)$ is coercive on $\widetilde{H}^{1}(D)$ and $c_k (\cdot , \cdot)$ is compact and analytic with respect to the wavenumber $k$ provided that $k^2 \in \mathbb{C} \setminus \brac{\lambda_j (D)}_{j=1}^{\infty}$.
\end{lemma}
\begin{proof}
    We begin by showing a useful estimate for the boundary integral. By Green's 1st Theorem and the fact that $v=w$ on $\partial D$, we have that
    $$I:= \int_{\partial D} \overline{w}T_0 w \, \text{d}s = \int_{\partial D} \overline{v} \partial_{\nu} v \, \text{d}s = \int_{D} |\nabla v|^2 \, \text{d}s = \| \nabla v\|^{2}_{L^{2}(\Omega)}.$$
    However, it also holds that
    $$I = \int_{\partial D} \overline{w} \partial_{\nu}v \, \text{d}s = \int_{D} \nabla v \cdot \nabla \overline{w} \, \text{d}x \leq  \| \nabla v\|_{L^{2}(D)} \| \nabla w\|_{L^{2}(D)} = \sqrt{I} \| \nabla w\|_{L^{2}(D)}. $$
Thus, we have the estimate $I \leq \| \nabla w\|^{2}_{L^{2}(D)}$. 

Now, suppose by contradiction that $b(\cdot , \cdot)$ is not coercive on $ \widetilde{H}^{1}(D)$. Thus, there exists a sequence $\{ w_m\}_{m \in \mathbb{N}}$ where $|b (w_m , w_m)|\rightarrow 0$ as $m \rightarrow \infty$ such that $\| w_m\|_{\widetilde{H}^{1}(D)} = 1$ for all $m \in \mathbb{N}$. By our above estimate, 
    $$|b(w_m , w_m)| \geq  \|w_m\|^{2}_{L^{2}(D)} + \min\{ \gamma_{\text{min}} , \mu_{\text{min}}\} \| w_m \|^{2}_{H^{1}(\partial D)}.$$
     Since the sequence $\{ w_m\}_{m \in \mathbb{N}}$ is bounded, then $w_m \rightharpoonup w$ as $m \rightarrow \infty$, i.e. weakly converges, in $\widetilde{H}^{1}(D)$. Since $|b (w_m , w_m)| \rightarrow 0$ as $m \rightarrow \infty$, then
    $$w_m \rightarrow 0 \quad \text{in} \enspace L^2(D) \qquad \text{and} \qquad w_m \big \rvert_{\partial D}\rightarrow 0 \quad \text{in} \enspace H^{1}(\partial D).$$
    Furthermore, let $\gamma_{\text{max}} = \| \gamma\|_{L^\infty (\partial D)}$ and $\mu_{\text{max}}= \sup\limits_{|\xi|=1}\| \overline{\xi} \cdot \mu  \xi \|_{L^\infty (\partial D)}$ then we have 
    \begin{align*}
        \| \nabla w_m\|^{2}_{L^{2}(D)} &= b (w_m , w_m) - \int_{D} |w_m|^2 + \int_{\partial D} \overline{w}_m T_0 w_m \, \text{d}s \\
         & \hspace{0.8in}  - \int_{\partial D} \mu |\nabla_{\partial D} w_m|^2 + \gamma  |w_m|^2 \, \text{d}s \\
        &\leq |b (w_m , w_m)| +  \| w_m\|^{2}_{L^{2}(D)} +   \| w_m\|_{H^{1/2}( \partial D)} \| T_0 w_m\|_{H^{-1/2}( \partial D)} \\ 
        &\hspace{.8in} + \max\{ \mu_{\text{max}} , \gamma_{\text{max}}\}  \| w_m\|^{2}_{H^{1}(\partial D)} \\
        &\leq b_k (w_m , w_m) + \| w_m\|^{2}_{L^{2}(D)} + C  \| w_m\|^{2}_{H^{1}(\partial D)}
    \end{align*}
    where the last inequality is obtained from the fact that $T_0$ is bounded and the continuous embedding of $H^1(\partial D)$ into $H^{1/2}( \partial D)$. Thus, $w_m \rightarrow 0$ as $m \rightarrow \infty$ in $\widetilde{H}^{1}(D)$ which contradicts the unit norm assumption on the sequence. Thus, $b(w, \phi)$ is indeed coercive on $\widetilde{H}^{1}(D)$. Since $b(\cdot , \cdot)$ is  bounded, there exists an invertible operator $B : \widetilde{H}^{1}(D) \rightarrow \widetilde{H}^{1}(D)$ such that 
    $$b(w , \phi) = (B w,\phi)_{\widetilde{H}^{1}(D)} \quad \text{ for all} \quad w,\phi \in \widetilde{H}^{1}(D).$$
    
    Furthermore, the compactness of $c_k(\cdot , \cdot)$ follows from the fact that $T_k - T_0$ is compact and the compact embedding of $H^{1}(D)$ into $L^{2}(D)$. Since $c_k(\cdot , \cdot)$ is also bounded, there exists compact operator $C_k :\widetilde{H}^{1}(D) \rightarrow \widetilde{H}^{1}(D)$ where 
   $$c_{k}(w , \phi) = (C_k w, \phi )_{\widetilde{H}^{1}(D)}  \quad \text{ for all} \quad w,\phi \in \widetilde{H}^{1}(D).$$  
The operator $T_k$ is analytic provided that $k^2$ is not a Dirichlet eigenvalue of $- \Delta$. Thus, $C_k$ is analytic with respect to $\displaystyle{k^2 \in \mathbb{C} \setminus \{ \lambda_j (D)\}^{\infty}_{j=1}}$. This proves the claim, since $B$ is invertible along with  $C_k$ being compact and analytic with respect to $k^2 \in \mathbb{C} \setminus \{ \lambda_j (D)\}^{\infty}_{j=1}$.
\end{proof}

Now, in order to prove discreteness of the TEs, we must show that there exists some value of $k$ such that $k^2 \in \mathbb{C} \setminus \{ \lambda_j (D)\}^{\infty}_{j=1}$ that is not a TE. The next result gives the discreteness for this TE problem. 

\begin{theorem}
    If the refraction index $\mathrm{Re}(n)>1$, then the set of transmission eigenvalues satisfying \eqref{v-w-in-D}--\eqref{v-w-on-bD} is at most discrete.
\end{theorem} 
\begin{proof}
    Let $(v,w) \in \widetilde{H}^{1}(D) \times \widetilde{H}^{1}(D)$ satisfy \eqref{v-w-in-D}--\eqref{v-w-on-bD}. We first show an estimate of the pairing between $w$ and $T_\text{i} w$. By Green's 1st theorem,
    $$I:= \int_{\partial D} \overline{w} T_\text{i} w \, \text{d}s = \int_{\partial D} \overline{v} \partial_{\nu} v \, \text{d}s = \int_{D} |\nabla v|^2 + |v|^2 \, \text{d}x = \|v \|^{2}_{H^{1}(D)}.$$
    It also holds by Cauchy-Schwartz that
    $$I = \int_{\partial D} \overline{w} \partial_{\nu} v \, \text{d}s = \int_{D} \nabla v \cdot \nabla \overline{w} + v \overline{w} \, \text{d}x \leq \| v\|_{H^{1}(D)} \| w\|_{H^{1}(D)} = \sqrt{I} \| w\|_{H^{1}(D)}.$$
    Thus, we have the estimate $I \leq \| w\|^{2}_{H^{1}(D)}$. We now show that $k=\text{i}$ is not a TE. To this end, assume that $(v,w)$ satisfy \eqref{v-w-in-D}--\eqref{v-w-on-bD} for $k=\text{i}$. By taking the real part of \eqref{a-k} for $\phi =w$ and our aforementioned estimate, we have that
    \begin{align*}
        0 &= \mathrm{Re}\int_{D} |\nabla w |^2 + n |w|^2\, \text{d}x - \mathrm{Re}\int_{\partial D} \big( T_\text{i} w - \mathscr{B}(w) \big) \overline{w} \, \text{d}s \\
        &\geq \int_{D} \left(\mathrm{Re}(n)-1 \right) |w|^2 \, \text{d}x + \int_{\partial D} \mathrm{Re}(\mu) |\nabla_{\partial D} w| + \mathrm{Re}(\gamma) |w|^2 \, \text{d}s.
    \end{align*}
   By the fact that $\mathrm{Re}(n)>1$ we have that $w=0$ in $D$ as well as $w = \partial_{\nu} w=0$ on $\partial D$. By \eqref{v-w-in-D}--\eqref{v-w-on-bD}, we have that $v$ satisfies
    $$(\Delta - 1) v = 0 \enspace \text{in} \enspace D \quad \text{and} \quad v=\partial_{\nu} v= 0 \enspace \text{on} \enspace \partial D.$$
By Holmgren's Theorem, we also have that $v=0$ in $D$. Therefore, if $k=\text{i}$, we have that $(w,v)=(0,0)$ and are not eigenfunctions. Thus, when $\mathrm{Re}(n)>1$, $k=\text{i}$ is not a TE. By the Analytic Fredholm Theorem, the set of TE is at most discrete.
\end{proof}

The analysis in the previous case will not work for the case when $0<\mathrm{Re}(n)<1$. This is because, just as in \cite{Hughes}, we would have to change our assumptions on the sign of the boundary parameters. In the aforementioned manuscript, the sign of the boundary parameter does not affect the well--posedness for the direct scattering problem, which is not the case here. To avoid having contradictory assumptions in different sections, we provide a non--existence result for the transmission eigenvalue in the case of complex--valued coefficients. 

To begin, we define the Hilbert space 
$$X(D) =\brac{ (\phi,\psi ) \in \widetilde{H}^{1}(D) \times \widetilde{H}^{1}(D)  \enspace \text{such that} \enspace \phi=\psi \enspace \text{on $\partial D$}}$$
with the associated graph norm/inner--product. Therefore, we have that the eigenfunctions $(w,v) \in X(D)$ satisfy 
$$
\int_{D} \nabla w \cdot \nabla \overline{\phi} - k^2 n w \overline{\phi} \,  \text{d} x = \int_{\partial D}   \overline{\phi}  \partial_{\nu} w  \, \text{d}s
$$
and 
\begin{align*}
\int_{D} \nabla v \cdot \nabla \overline{\psi} - k^2  v \overline{\psi} \, \text{d}x &= \int_{\partial D}   \overline{\psi}  \partial_{\nu} v  \, \text{d}s \\
	&=  \int_{\partial D}   \overline{\psi}  \partial_{\nu} w +   \overline{\psi} \mathscr{B}(w) \, \text{d}s 
\end{align*}
for any $(\phi,\psi ) \in X(D)$. 
Since the $\phi=\psi$ on $\partial D$ we have that
\begin{align}\label{int-id-4noeig}
 \int_{D} \nabla v \cdot \nabla \overline{\psi} - k^2  v \overline{\psi} \, \text{d}x = \int_{D} \nabla w \cdot \nabla \overline{\phi} - k^2 n w \overline{\phi} \, \text{d}x+ \int_{\partial D} \mu \nabla_{\partial D}w \cdot \nabla_{\partial D} \overline{\phi} +\gamma w \overline{\phi}  \, \text{d}s.
\end{align}
With this integral identity, we can prove the following result. 
\begin{theorem}
    If the coefficients are complex--valued such that  
    $$\mathrm{Im}(n) > 0, \quad - \mathrm{Im}(\gamma) \geq 0 \quad \text{and} \quad - \overline{\xi} \cdot \mathrm{Im}(\mu )  \xi \geq 0$$
for all $\xi \in \mathbb{C}^{d-1}$ for almost every $x \in \partial D$.
    Then there are no real--valued transmission eigenvalues. 
\end{theorem} 
\begin{proof}
In order to prove the claim, we consider \eqref{int-id-4noeig} where we take  $(\phi,\psi )=(w,v)$. This gives the equality 
$$ 
\int_{D} |\nabla v|^2  - k^2  |v|^2 \, \text{d}x = \int_{D} |\nabla w|^2- k^2 n |w|^2 \, \text{d}x + \int_{\partial D} \mu |\nabla_{\partial D}w |^2 +\gamma |w|^2 \, \text{d}s.
$$
Now, we take the imaginary part of the above expression to obtain 
$$ 
-k^2 \int_{D} \text{Im}(n) |w|^2 \, \text{d}x =  - \int_{\partial D} \text{Im}( \mu ) |\nabla_{\partial D}w |^2 + \text{Im}(\gamma) |w|^2 \, \text{d}s.
$$
By the assumptions on the coefficients, we have that $w=0$ in $D$ since the lefthand side is non--positive and the righthand side is non--negative.  Just as before, this implies that $v=0$ in $D$, which implies that for any real--valued wave, $k$ is not a transmission eigenvalue, proving the claim. 
\end{proof}

\subsection{Some Numerical Results for TE}\label{numerical-TEV}
Here, we will provide some numerical estimates for the above TE problem. To this end, we assume the coefficients are real-valued where we let $n>1$ be a given constant as well as $\mu$ and $\gamma$. The transmission problem 
\begin{eqnarray*}
\Delta v+k^2v=0\,,\quad \Delta w+k^2 nw=0
\end{eqnarray*}
in $D$ with boundary conditions $v-w=0$ and $\partial_\nu v-\partial_\nu w-\mathscr{B}(w)=0$ can be solved with boundary integral equation using the same ansatz as in the Appendix given in Section \ref{bie}. Precisely, we use the single-layer ansatz
\begin{eqnarray*}
 v(x)&=&\mathrm{SL}_{k}\varphi(x)\,,\quad x\in D\,,\\
 w(x)&=&\mathrm{SL}_{k\sqrt{n}}\psi(x)\,,\quad x\in D\,.
\end{eqnarray*}
Using the first boundary condition yields
\begin{eqnarray*}
\mathrm{S}_{k}\varphi(x)-\mathrm{S}_{k\sqrt{n}}\psi(x)=0\,,\quad x\in \partial D\,.
\end{eqnarray*}
Using the second boundary condition gives
\begin{eqnarray*}
\left(\frac{1}{2}I+\mathrm{D}^\top_{k}\right)\varphi(x)-\left(\frac{1}{2}I+\mathrm{D}^\top_{k\sqrt{n}}+\gamma \mathrm{S}_{k\sqrt{n}}-\mu \mathrm{T}_{k\sqrt{n}}\right)\psi(x)=0\,,\quad x\in \partial D\,.
\end{eqnarray*}
The last two equations can be written as a non-linear eigenvalue problem for the eigenvalue $k$ and the eigenfunction $z=(\varphi,\psi)^\top$ as
\begin{eqnarray*}
M(k)z=0\,,
\end{eqnarray*}
with the obvious definition of $M(k)$. After discretization of this problem via boundary element collocation method, a non-linear eigenvalue solver such as the contour integral method by \cite{beyn} is used to compute the needed eigenvalues $k$ within a chosen contour lying in the complex-plane. If they exist, this gives us the opportunity to compute pure complex-valued transmission eigenvalues.  

\subsubsection{Transmission Eigenvalues for a disk}
The transmission eigenvalues of a disk with radius $R>0$ centered at the origin can be numerically computed by finding the roots of a certain determinant. Precisely, using a similar ansatz as in Section \ref{sov} we have that 
$$v(r, \theta ) = \sum_{|p|=0}^{\infty} \text{v}_{p} J_{p}(kr) \text{e}^{\text{i}p\theta} \quad \text{and} \quad w(r , \theta) = \sum_{|p|=0}^{\infty}  \text{w}_{p} J_{p} (k \sqrt{n} r) \text{e}^{\text{i}p\theta}.$$
The series solutions already satisfy the equations in the scatterer $D$. By applying the boundary conditions we obtain that $k$ is a TE provided that it is a zero of the determinant 
\begin{eqnarray}
    \mathrm{det}\begin{pmatrix}
    J_p(kR) & -J_p\left(k\sqrt{n}R\right)\\
    k J_p'(kR) & -\left(k\sqrt{n}J_p'\left(k\sqrt{n}R\right)+\left(\gamma  +\frac{\mu p^2}{R^2}\right) J_p\left(k\sqrt{n}R\right)\right)
    \end{pmatrix}
    \label{2D}
\end{eqnarray}
for arbitrary but fixed $p\in \Z$, where $R>0$, $n>1$, $\mu$, and $\gamma$ are given parameters.

\subsubsection{Numerical \textcolor{black}{results} via BIEs}
First, we compute the transmission eigenvalues for a disk with radius $R=2$ using the parameters $n=4$, $\mu=2$, and $\gamma=1$ using both the integral equation method and the series expansion method to show that both method deliver the same results. Note that we only compute the eigenvalues within an interesting set such as $\{z\in\mathbb{C}\,:\,|z-1.25|\leq 0.35\}\subset \mathbb{C}$.
\begin{table}[!ht]
\centering
 \begin{tabular}{r|r|r|}
 $p$ & TE (series) & BEM\\
 \hline
2 & $1.081995004204943$ & $1.082779369684411 - 0.000378339595974\mathrm{i}$\\
2 & $1.081995004204943$  & $1.082779369718667 - 0.000378337460393\mathrm{i}$ \\
3 & $1.444057126606098 $ &$1.445031960126334 - 0.000027829895197\mathrm{i}$\\
3 & $1.444057126606098$ &$1.445031958798628 - 0.000027829450002\mathrm{i}
$\\
1 & $1.567008428331221$ &$1.567915072883429 - 0.000599790601235\mathrm{i}$\\
1 & $1.567008428331221$ &$1.567913270599214 - 0.000601136192057\mathrm{i}$\\
0 & $1.223227533499797 - 0.236035304541013\mathrm{i}$ &$1.224593675536148 - 0.237072460045154\mathrm{i}$\\
0 & $1.223227533499797 + 0.236035304541013\mathrm{i}$ &$1.223647971805621 + 0.235866668217871\mathrm{i}$\\
 \hline
 \end{tabular}
 \caption{\label{TE1}
 Transmission eigenvalue for a disk with radius $R=2$ and parameters $n=4$, $\mu=2$, and $\gamma=1$ using boundary element collocation method and series expansion.}
\end{table}
As we can see in Table \ref{TE1}, we obtain six real transmission eigenvalues (counting multiplicity) and one pair of pure complex-valued eigenvalues hence showing that also complex-valued transmission eigenvalues exist. The transmission eigenvalues agree to three digits accuracy using $60$ collocation nodes within the boundary element collocation method and the parameters $N=24$ and $\ell=20$ within the Beyn method \cite{beyn}. Of course, we are able to compute many more pure complex transmission eigenvalues such as
\[2.086\pm 0.228\mathrm{i}\,,\quad 2.669\pm 0.251\mathrm{i}\,,\quad 2.749\pm 0.277\mathrm{i}\,,\quad 3.210\pm 0.127\mathrm{i}\]
which are located close to the real axis, an interesting fact on its own.

Note that the eigenfunctions that are radial symmetric (they do not depend on the angle $\theta$) are automatically satisfying the tangential derivative condition. Hence, for $p=0$ the eigenvalues for the new problem are the same as for $\mu=0$.

Next, we compute the transmission eigenvalues for an ellipse with semi-axis $a=1$ and $b=0.9$ as well as for $a=1$ and $b=0.8$ using the same parameters as before with the boundary integral equation method with 120 collocation points and in combination with the Beyn algorithm. We obtain the following results as shown in Table \ref{TE2}. 
\begin{table}[!ht]
\centering
 \begin{tabular}{r|r|r|}
 disk $R=1$ & ellipse $(1,0.9)$ & ellipse $(1,0.8)$\\ 
 \hline
 $0.683-0.001\mathrm{i}$ & $0.708+0.000\mathrm{i}$& $0.737+0.005\mathrm{i}$\\
 $1.434-0.001\mathrm{i}$ & $1.513+0.000\mathrm{i}$& $1.605+0.002\mathrm{i}$\\
 $1.434-0.001\mathrm{i}$ & $1.522+0.000\mathrm{i}$& $1.629+0.002\mathrm{i}$\\
 $2.290-0.000\mathrm{i}$ & $2.405-0.014\mathrm{i}$& $2.486-0.055\mathrm{i}$\\
 $2.290-0.000\mathrm{i}$ & $2.427-0.000\mathrm{i}$& $2.597-0.000\mathrm{i}$\\
 $3.016-0.000\mathrm{i}$ & $3.133-0.021\mathrm{i}$& $3.203-0.048\mathrm{i}$\\
 $3.016-0.000\mathrm{i}$ & $3.164-0.010\mathrm{i}$& $3.323-0.021\mathrm{i}$\\
 $3.135-0.001\mathrm{i}$ & $3.294+0.016\mathrm{i}$& $3.555+0.026\mathrm{i}$\\
 $3.135-0.001\mathrm{i}$ & $3.409+0.007\mathrm{i}$& $3.774+0.011\mathrm{i}$\\
 \hline
 \end{tabular}
 \caption{\label{TE2}
 Transmission eigenvalues for a disk with radius $R=1$, an ellipse with semi-axis $(1,0.9)$, and an ellipse with semi-axis $(1,0.8)$ and parameters $n=4$, $\mu=2$, and $\gamma=1$ using boundary element collocation method.}
\end{table}
We also list the results for a disk of radius $R=1$ to show how the transmission eigenvalues change when changing one of the half-axis of the ellipse. We remark that the real-part of the transmission eigenvalues increase when the semi-axis is decreased (the area of the obstacle decreases in this case). Also note that in Table \ref{TE2}, the TEs are most likely real--valued and the small imaginary part is due to numerical error.

\section{Conclusions}\label{conclusion}
Here, we studied the direct and inverse scattering problem for an isotropic material with a second-order boundary condition. We rigorously proved the well-posedness of the direct scattering problem. For the inverse scattering problem, we derived a direct sampling method with an imaging functional that is stable with respect to noise. In two dimensions, we provided several numerical examples where we recover circular and non-circular scatterers for various levels of noise.
We used separation of variables, Lippmann-Schwinger integral equations, and boundary integral equations to generate the data for these reconstructions. Our examples display a high level of resilience to noise. We also proved that the associated transmission eigenvalues form at most a discrete set. We then provide numerical examples in two dimensions to validate our theoretical result on discreteness. Future directions of this project include proving the existence of the transmission eigenvalues. Also, one may investigate the monotonicity and convergence of the transmission eigenvalues with respect to the boundary parameters. For future projects, it would be interesting to derive a more accurate scheme for the TE and forward scattering problem. One could also consider the case for limited aperture data. \\


\noindent{\bf Acknowledgments:} The research of I. Harris is currently partially supported by the NSF DMS Grants 2208256 and 2509722. G. Granados began this project at Purdue University where he received partial support from the NSF DMS Grant 2107891. Currently, G. Granados acknowledges support from the NSF RTG DMS Grant 2135998.  

{\color{black}
\section{Appendix}
\subsection{Separation of Variables for the Direct Problem}\label{sov}
In this subsection, we will derive the solution to the direct scattering problem via separation of variables that was used in Section \ref{numerical-section}. We assume that $\partial D=R(\text{cos}(\theta) , \text{sin}(\theta))$ and $\partial \Omega=(\text{cos}(\theta) , \text{sin}(\theta))$ for some $0<R < 1$ where $\theta \in [0 , 2\pi )$. Using separation of variables in polar coordinates, we can derive an analytical formula for the Cauchy data on $\partial \Omega$. To this end, we consider the scattering problem \eqref{bvp}--\eqref{2nd-bc} with constant coefficients $n$, $\mu$, and $\gamma $. Therefore, we have that 
\begin{center}
$\Delta u^s + k^2 u^s = 0  \enspace \text{in} \enspace  \mathbb{R}^d \setminus \overline{B(0,R)} \quad \text{and} \quad \Delta u + k^2 n u = 0 \enspace \text{in} \enspace B(0,R)$ \vspace{.15in}

$(u^s + u^i)_{+} = u \quad \text{and} \quad \partial_{r}(u^s + u^i)_{+} = \partial_{r} u - \mu \frac{\text{d}^2}{\text{d}s^2} u + \gamma u \enspace \text{on} \enspace \partial B(0,R)$
\end{center}
with the Sommerfeld radiation condition as $r \rightarrow \infty$ for the scattered field $u^s$. We can express the incident field $u^i(x , \hat{y} ) = \text{e}^{\text{i}kx \cdot \hat{y}}$ using the Jacobi-Anger expansion
$$\text{e}^{\text{i}kx \cdot \hat{y}} = \sum_{|p|=0}^{\infty} \text{i}^p J_p (kr) \text{e}^{\text{i}p(\theta - \phi)}$$
where $x = r (\text{cos}(\theta) , \text{sin}(\theta))$, $\hat{y} = (\text{cos}(\phi) , \text{sin}(\phi))$, and $J_p$ is the Bessel function of the first kind of order $p$. We make the ansatz that the scattered field $u^s$ and total field $u$ can be expressed as
$$u^{s}(r, \theta ) = \sum_{|p|=0}^{\infty} \text{i}^p \text{u}^{s}_{p} H^{(1)}_{p}(kr) \text{e}^{\text{i}p(\theta - \phi)} \quad \text{and} \quad u(r , \theta) = \sum_{|p|=0}^{\infty} \text{i}^p \text{u}_{p} J_{p} (k \sqrt{n} r) \text{e}^{\text{i}p(\theta - \phi)},$$
where $H^{(1)}_{p}$ are the Hankel functions of the first kind of order $p$. We are able to determine the coefficients $\text{u}^{s}_{p}$ and $\text{u}_{p}$ by using the boundary conditions at $r = R$. Thus, the coefficients $ \text{u}^{s}_{p}$ and $\text{u}_{p}$ satisfy the $2 \times 2$ linear system
\begin{align} 
\begin{pmatrix}
H^{(1)}_{p}(kR) & - J_{p}(k \sqrt{n} R) \\
k H^{(1)\prime}_{p}(kR) & - k \sqrt{n} J^{\prime}_p (k \sqrt{n} R ) -  \big(\mu \frac{p^2}{R^2}  + \gamma\big) J_p (k \sqrt{n} R )
\end{pmatrix} 
\begin{pmatrix}
\text{u}^{s}_{p} \\
\text{u}_{p}
\end{pmatrix}
= \begin{pmatrix}
-J_p (kR) \\
-k J_{p}^{\prime}(kR)
\end{pmatrix} \label{sov-mat}
\end{align}
for all $p \in \mathbb{Z}$. We use Cramer's Rule to compute the coefficients $\text{u}^s_p$, where we obtain
\[ 
u^s_p = \dfrac{\text{det}A_x}{\text{det}A} \quad \text{where} \quad
A= \begin{pmatrix}
H^{(1)}_{p}(kR) & - J_{p}(k \sqrt{n} R) \\
k H^{(1)\prime}_{p}(kR) & - k \sqrt{n} J^{\prime}_p (k \sqrt{n} R ) -  \big(\mu \frac{p^2}{R^2}  + \gamma\big) J_p (k \sqrt{n} R )
\end{pmatrix} \]
and
\[A_x = \begin{pmatrix}
-J_p (kR) & - J_{p}(k \sqrt{n} R) \\
-k J_{p}^{\prime}(kR) & - k \sqrt{n} J^{\prime}_p (k \sqrt{n} R ) -  \big(\mu \frac{p^2}{R^2}  + \gamma\big) J_p (k \sqrt{n} R )
\end{pmatrix}.\]
With this, we have that the Cauchy data can be approximated by
$$u^s (r , \theta) \approx \sum_{|p|=0}^{15} \text{i}^p \text{u}^{s}_{p} H^{(1)}_{p}(kr) \text{e}^{\text{i}p(\theta - \phi)} \quad \text{and} \quad \partial_{r} \text{u}^s (r , \theta) \approx k \sum_{|p|=0}^{15} \text{i}^p u^{s}_{p} H^{(1)'}_{p}(kr) \text{e}^{\text{i}p(\theta - \phi)}.$$
This is used for the reconstructions in Figures \ref{fig0}, \ref{fig1}, and \ref{fig3}.

\subsection{Boundary Integral Equations  for the Direct Problem}\label{bie}
Here, we will discuss the system of boundary integral equations to compute the Cauchy data used in Section \ref{numerical-section}. To this end, we assume that the scatterer $D$ is given by a simply-connected bounded domain with boundary $\partial D$ of class $C^{4,\alpha}$ with a unit normal $\nu$ pointing in the exterior. The parameter $n>1$ denotes the given constant index of refraction, $k>0$ the given wave number, and $\gamma$ and $\mu$ are given material parameter. Note, that the far--field pattern for the separation of variables solution is given by
\begin{eqnarray}
 u^\infty(\hat{x}, \hat{y})=\frac{4}{\mathrm{i}}\sum_{|p|= 0}^\infty \text{u}^s_p \mathrm{e}^{\mathrm{i}p(\theta -\phi)}\,,  \label{farfieldseries}
\end{eqnarray}
for $\hat{x} = (\cos(\theta), \sin(\theta) )$ and $\hat{y} = (\cos(\phi), \sin(\phi) )$, where the coefficients $\text{u}^s_p$ are given by solving \eqref{sov-mat}. 

We now derive a system of boundary integral equations to solve for the scattered field and the far--field. Therefore, in order to solve that forward problem, we make the single-layer ansatz
\begin{eqnarray}
 u^s(x)&=&\mathrm{SL}_{k}\varphi(x)\,,\quad x\in\mathbb{R}^2 \backslash \overline{D}\,,\label{start1}\\
 u(x)&=&\mathrm{SL}_{k\sqrt{n}}\psi(x)\,,\quad x\in D\,,
 \label{start2}
\end{eqnarray}
where the single layer operator is defined by
\begin{eqnarray*}
\mathrm{SL}_{\tau}\phi(x)&=&\int_{\partial D}\Phi_\tau(x,y)\phi(y)\,\mathrm{d}s(y)\,,\quad  x\notin \partial D\,,
\end{eqnarray*}
with $\Phi_\tau (x,y)$ the fundamental solution of the Helmholtz equation in two dimensions for the fixed wave number $\tau$. Now, we let $x$ approach the boundary in (\ref{start1}) and (\ref{start2}). With the jump conditions (see \cite[Theorem 3.1]{colton-kress}), we obtain
\begin{eqnarray}
 u^s(x)&=&\mathrm{S}_{k}\varphi(x)\,,\quad x\in\partial D\,,\label{start1a}\\
 u(x)&=&\mathrm{S}_{k\sqrt{n}}\psi(x)\,,\quad x\in \partial D\,,
 \label{start2a}
\end{eqnarray}
where the single-layer boundary integral operator is defined by
\begin{eqnarray*}
    \mathrm{S}_{\tau}\phi(x)&=&\int_{\partial D}\Phi_\tau (x,y)\phi(y)\,\mathrm{d}s(y)\,,\quad x\in \partial D\,.
\end{eqnarray*}
Taking the normal derivative and letting $x$ approach the boundary along with the jump conditions (\cite[Theorem 3.1]{colton-kress}), we obtain
\begin{eqnarray}
 \partial_{\nu(x)} u^s(x)&=&\left(-\frac{1}{2}I+\mathrm{D}^\top_{k}\right)\varphi(x)\,,\quad x\in\partial D\,,\label{start1b}\\
 \partial_{\nu(x)} u(x)&=&\left(\frac{1}{2}I+\mathrm{D}^\top_{k\sqrt{n}}\right)\psi(x)\,,\quad x\in \partial D\,,
 \label{start2b}
\end{eqnarray}
where the normal derivative of the single-layer boundary integral operator is defined by
\begin{eqnarray*}
    \mathrm{D}^\top_{k}\phi(x)&=&\int_{\partial D}\partial_{\nu(x)}\Phi_k(x,y)\phi(y)\,\mathrm{d}s(y)\,,\qquad x\in \partial D\,.
\end{eqnarray*}
We finally define the operator  
\begin{eqnarray}
    \mathrm{T}_{k}\phi(x)&=&\frac{\text{d}^2}{\text{d}s^2}\int_{\partial D}\Phi_k(x,y)\phi(y)\,\mathrm{d}s(y)\,,\qquad x\in \partial D\,,
    \label{troubleop}
\end{eqnarray}
where ${\text{d}}/{\text{d}s}$ denotes the tangential derivative.
Using the boundary conditions yield the two equations for the two unknown densities on the boundary (see (\ref{start1a}) and (\ref{start2a}) as well as (\ref{start1b}) and (\ref{start2b}))
\begin{align}
\mathrm{S}_{k}\varphi(x)-\mathrm{S}_{k\sqrt{n}}\psi(x)=&-u^i , \nonumber\\
\left(-\frac{1}{2}I+\mathrm{D}^\top_{k}\right)\varphi(x)-\left(\frac{1}{2}I+\mathrm{D}^\top_{k\sqrt{n}}+\gamma \mathrm{S}_{k\sqrt{n}}-\mu \mathrm{T}_{k\sqrt{n}}\right)\psi(x)=&-\partial_\nu u^i .
\label{system}
\end{align}
After we solve the $2\times 2$ system for $\varphi$ and $\psi$ on the boundary, we can compute the scattered field using (\ref{start1}) at any point in $\mathbb{R}^2\backslash \overline{D}$ or the far-field of $u^s$ by computing
\begin{eqnarray}
u^\infty(\hat{x})=\mathrm{S}_k^\infty \varphi (\hat{x}), \quad \text{ where } \quad \mathrm{S}^\infty_k\phi(\hat{x})=\int_{\partial D} \mathrm{e}^{-\mathrm{i}k \hat{x}\cdotp y}\phi(y)\,\mathrm{d}s(y)\,,\quad \hat{x}\in \mathbb{S}^1\,.
\label{farfieldexpression}
\end{eqnarray}
Then, the Cauchy data is given by 
$$u^s(x)=\mathrm{SL}_{k}\varphi(x) \quad \text{ and  } \quad \partial_{\nu(x)} u^s(x)=\partial_{\nu(x)}\mathrm{SL}_{k}\varphi(x)\,,\quad x \in \partial \Omega$$ 
where $\varphi = \varphi( \cdot \, ; \hat{y})$ solves \eqref{system}. 

In order to check the accuracy for solving the system we let $N_f$ be the number of faces used in the boundary element collocation method. Since we use a $2\pi$ parametrization of the boundary in the form $z(t)=(z_1(t),z_2(t))$ with $t\in [0,2\pi]$, the operator $T$ given by
\begin{eqnarray*}
\mathrm{T}_{k}\phi(x)&=&\frac{\mathrm{d}^2}{\mathrm{d}s^2}\int_{\partial D}\Phi_k(x,y)\phi(y)\,\mathrm{d}s(y)\\
&=&\frac{\mathrm{d}^2}{\mathrm{d}s^2}S_k\phi(x)\,,\qquad x\in \partial D\,,
\end{eqnarray*}
can be rewritten as (see also \cite[p. 6]{cakonikress})
\begin{eqnarray*}
&&\frac{1}{\|z'(t)\|}\frac{\mathrm{d}}{\mathrm{d}t}\left(\frac{1}{\|z'(t)\|}\frac{\mathrm{d}}{\mathrm{d}t}S_k\phi(z(t))\right)\\
&=&\frac{1}{\|z'(t)\|}\left[\left(\frac{\mathrm{d}}{\mathrm{d}t}\frac{1}{\|z'(t)\|}\right)\cdotp \left(\frac{\mathrm{d}}{\mathrm{d}t}S_k\phi(z(t))\right)+\frac{1}{\|z'(t)\|}\frac{\mathrm{d}^2}{\mathrm{d}t^2}S_k\phi(z(t))\right]\,,
\end{eqnarray*}
where $x=z(t)\in \partial D$.
Using centered finite differences for the first- and second-order derivative, which are of second-order accuracy yield the approximations
\begin{eqnarray*}
\frac{\mathrm{d}}{\mathrm{d}t}S_k\phi(z(t))&\approx &\frac{S_k\phi(z(t+h))-S_k\phi(z(t-h))}{2h}\\
\frac{\mathrm{d}^2}{\mathrm{d}t^2}S_k\phi(z(t))&\approx &\frac{S_k\phi(z(t+h))-2S_k\phi(z(t))+S_k\phi(z(t-h))}{h^2}\\
\end{eqnarray*}
where $h>0$ small is given by the user. The operator $S_k$ is then approximated in the usual way. We will check the error of the far-field pattern by comparing the separation of variable solution of \eqref{system}. To this end, we define the far-field matrices 
$${\bf F}_k \in \mathbb{C}^{64 \times 64} \quad \text{ and} \quad {\bf F}_k^{(N_f)}\in \mathbb{C}^{64 \times 64}$$ 
containing values of the the far-field data for $64$ equidistant incident and observation directions obtained with the series expansion (\ref{farfieldseries}) and with the boundary element collocation method using $N_f$ faces using (\ref{farfieldexpression}), respectively. The absolute error is defined as
$$\varepsilon_k^{(N_f)} :=\max\limits_{1\leq i,j \leq 64}  \left|{\bf F}_k (i,j) -{\bf F}_k^{(N_f)}(i,j) \right| ,$$
where $k$ is a given wave number. Note that the absolute error also depends on the physical parameter $n$, $\mu$, and $\gamma$. 

In Table \ref{tablefarfield}, we provide the computed absolute error in the far-field data, where the scatterer $D=B(0,2)$ and the physical parameters $\gamma=1$, $\mu=1$, and $n=4$. We check the error for multiple wave numbers given by $k=2$, $k=4$, and $k=6$. From our calculations we observe, reasonable accurate results using $240=3N_f$ collocation nodes. We discretize the tangential derivative that is only accurate up to two significant digits, which limits the accuracy of the collocation method with $h=0.01$ for smaller wave numbers. 
\begin{table}[H]
\centering
 \begin{tabular}{r|r|r|r|}
  $N_f$ & $\varepsilon_2^{(N_f)}$ & $\varepsilon_4^{(N_f)}$ & $\varepsilon_6^{(N_f)}$ \\
  \hline 
10 (\phantom{1}30)& 0.17048 & 2.76147 & 30.40369\\
20 (\phantom{1}60) & 0.02634 & 0.16304 & 0.77639\\
40 (120) & 0.00495 & 0.02392 & 0.14933\\
80 (240) & 0.00137 & 0.00375 & 0.02283\\
160 (480) & 0.00272 & 0.00281 & 0.00586\\
\hline
 \end{tabular}
 \caption{\label{tablefarfield}Absolute error of the far-field with 64 equidistant incident and observation directions for the disk with radius $R=2$ and the physical parameters $\gamma=1$, $\mu=1$, and $n=4$ for varying number of faces (collocation nodes). The wave numbers are $k=2$, $k=4$, and $k=6$, respectively.}
\end{table}
This approximation of the Cauchy data is used in Figures \ref{fig6}--\ref{fig8}.\\}


\end{document}